\documentclass[12pt]{amsart}
\usepackage{amsmath}

\usepackage{amssymb}
\usepackage[dvipsnames]{xcolor}
\usepackage{amsthm}

\usepackage{graphicx}
\usepackage{tikz}
\usepackage{tkz-graph}
\usetikzlibrary {arrows.meta}

  \newtheorem{theorem}{Theorem}[section]
  \newtheorem{lemma}[theorem]{Lemma}
   \newtheorem{proposition}[theorem]{Proposition}
  \newtheorem{definition}[theorem]{Definition}
  \newtheorem{corollary}[theorem]{Corollary}
  \newtheorem{example}[theorem]{{Example}}
  \newtheorem{problem}[theorem]{{Problem}}

  \def\ZZ{\mathbb{Z}}

  \def\Im{\mathrm{im}}

  \def\sgn{\mathrm{sign}}
  
 \def\img{\mathrm{im}}

  \def\Cube{{\normalfont {\textrm{Cube}}}}

     \def\inv{{\mathrm{INV}}}
       \def\flip{{\mathrm{FLIP}}}
     \def\HH{{\mathcal{H}}}

\title{Monophobic  Neighborhoods of Cubes}

\author[Greene]{Curtis Greene}

\address[Curtis Greene]
{
Haverford College, Haverford, PA 19041, USA
}
\email{cgreene@haverford.edu}

\author[Jamil]{Samira Sahar Jamil}

\address[Samira Sahar Jamil]
{
University of Notre Dame, Notre Dame, IN 46556 USA
}
\email{sjamil@nd.edu}

\thanks{
This material is based upon work done while the second author was in residence at the Simons Laufer Mathematical Sciences Institute (formerly MSRI) in Berkeley, California, during 2024, supported by the National Science Foundation under Grant No. DMS-1928930 and by the Alfred P. Sloan Foundation under Grant G-2021-16778.}

\begin{document}

\begin{abstract} 
We define a $d$-cube in a graph $G$ to be monophobic in $G$ if, whenever it supports a  $d$-face of a discrete singular $d+1$-cube, it must also support another face. 
In many cases, the presence of such cubes implies the nontriviality of discrete cubical homology and homotopy groups associated with $G$. We describe several large classes of
examples and characterize the property completely when $d=2$.
\end{abstract}

\date{ \today}

\maketitle

\section{Introduction}
This note describes a technique for proving that certain elements of discrete homotopy groups $A_d(G,v_0)$ and discrete cubical homology 
groups $H^{\Cube}_d(G)$ of graphs $G$ are nontrivial. The ambient discrete theory was originally developed for metric spaces in \cite{BKLW},
later refined for graphs in \cite{BCW}, \cite{BBLL} and subsequent papers. See \cite{BKLW} and \cite{BCW} for precise definitions and notation. The present note introduces the concept of {\em monophobicity}, identifying faces of a cube which are forced to appear more than once in the boundary of a higher dimensional singular cube.
 This extends a technique used in \cite{BGJW1} to prove the 
nontriviality of a certain homology group (\cite{BGJW1}, Theorem 6.1). While our original impetus was to find a natural setting for algebraic results in \cite{BGJW1}, much of this paper is about purely combinatorial properties of 
maps from discrete cubes to graphs and may also be understood in this light.

Section 2 summarizes the main results, and Section 3 contains proofs. Section 4 gives a complete characterization of 
monophobic neighborhoods in dimension $2$, and Section 5 explores preliminary steps toward a similar characterization in higher dimensions. Section 6 discusses the notion of natural cycles, which appear in several examples modeling discrete spheres.
Section 7 contains thoughts about possible future directions for this work.

\section{Summary of Results}

All of our graphs $G$ are finite and undirected. A {\em rigid $d$-cube} in $G$ is an induced subgraph $H$ isomorphic to $Q_d$, the standard discrete $d$-cube. 
A {\em singular $d$-cube in $G$} is a graph map 
$\sigma: Q_d \to G$, that is, a vertex map such that if $v_1 \sim v_2$ then $\sigma(v_1) \sim \sigma(v_2)$ or $v_1 = v_2$. A singular $d$-cube is {\em supported by $H$} if $\textrm{im}(\sigma)=V(H)$. The following property of rigid cubes will be used frequently:
\begin{lemma}
If $H$ is a rigid $d$-cube in $G$ and $\sigma$ is a singular $d$-cube supported by $H$, then $\sigma$ is uniquely determined by its values on any face of $H$.
\end{lemma}
For $d\ge 0$, the $\ZZ$-module generated by all singular $d$-cubes is denoted by $C_d(G)$, with associated homology groups $H_d(G)$  as defined in \cite{BCW}. 
A {\em based graph map} is a map $\alpha: (I_\infty^d, \partial I_\infty^d) \to (G,v_0)$, where $I_\infty^d \approx \ZZ^d$ is the {\em infinite subdivided $d$-cube}, $v_0\in V(G) $,   and $\alpha$ stabilizes to $v_0$ in all directions. An associated discrete homotopy theory (defined, e.g.,  in \cite{BKLW} and  \cite{BBLL}) yields homotopy groups $A_d(G,v_0), d\ge 0$.
\begin{definition}  If $H$ is a rigid $d$-cube in $G$, then $H$ is {\em monophobic in $G$} if, whenever a singular $(d+1)$-cube $\sigma\in C_{d+1}(G)$ contains a face supported by $H$, it must contain at least one more face supported by $H$. 
\end{definition}
Note that this is a {\em local} property: it depends only on the {\em $1$-neighborhood $N_G(H)$},   the induced subgraph of $G$ generated by $H$ and all adjacent vertices. Clearly $H$ is monophobic in $G$ if and only if $H$ is monophobic in $N_G(H)$, and in such a case we also describe   $N_G(H)$ as a 
{\em monophobic neighborhood of $H$}.  If $N_G(H)$ is not monophobic, we say that it is {\em monotolerant}.
\begin{example} \label{mpexamples}Examples of $2$-cubes $H$ (displayed in red) with  monophobic neighborhoods $N_G(H)$.
\end{example}
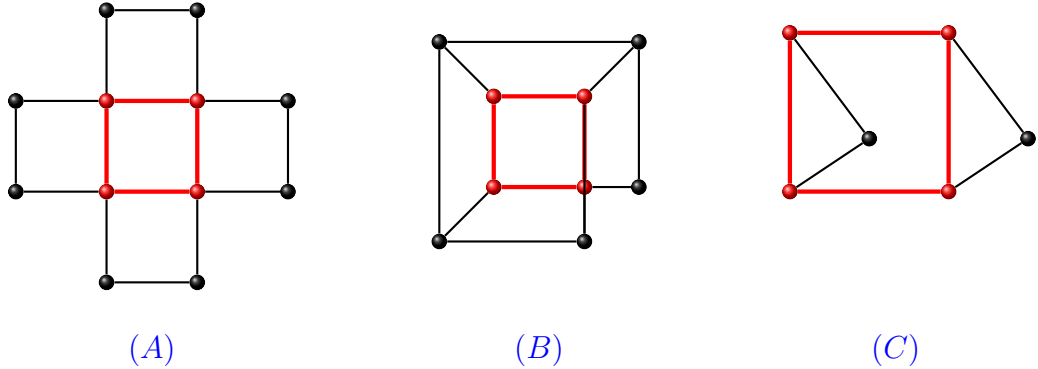
\begin{center}
 \qquad
\begin{tikzpicture}[scale=.6]  
\tikzset{VertexStyle/.style = {shape= circle,
ball color = red,inner sep= 0pt,outer sep= 0pt,minimum size=6 pt}}   
\Vertex[x=3 ,y=3,L=$$]{1}
   \Vertex[x=5 ,y=3,L=$$]{2}
   \Vertex[x=3 ,y=5,L=$$]{3}
   \Vertex[x=5 ,y=5,L=$$]{4}
   \tikzset{VertexStyle/.style = {shape= circle,
ball color =  black,inner sep= 0pt,outer sep= 0pt,minimum size=6 pt}}     
 \Vertex[x=3 ,y=1,L=$$]{5}
   \Vertex[x=5 ,y=1,L=$$]{6}
   \Vertex[x=7 ,y=3,L=$$]{7}
   \Vertex[x=7 ,y=5,L=$$]{8}
   \Vertex[x=5 ,y=7,L=$$]{9}
   \Vertex[x=3 ,y=7,L=$$]{10}
   \Vertex[x=1 ,y=5,L=$$]{11}
   \Vertex[x=1 ,y=3,L=$$]{12}
    \Edge[style={ultra thick},color=red](1)(2)  \Edge[style={ultra thick},color=red](2)(4)  \Edge[style={ultra thick},color=red](3)(4)
      \Edge[style={ultra thick},color=red](1)(3) 
      \Edge(1)(5)\Edge(2)(6)\Edge(5)(6)
       \Edge(2)(7)\Edge(4)(8)\Edge(7)(8)
       \Edge(4)(9)\Edge(3)(10)\Edge(9)(10)
       \Edge(3)(11)\Edge(1)(12)\Edge(11)(12)
       \node  [color=blue] at (4,-.5) {$(A)$};
  \end{tikzpicture}
  \qquad\qquad
 \begin{tikzpicture}[scale=.6]. 
\tikzset{VertexStyle/.style = {shape= circle,
ball color = red,inner sep= 0pt,outer sep= 0pt,minimum size= 6 pt}}   
   \Vertex[x=3 ,y=3+.2,L=$$]{1}
   \Vertex[x=5 ,y=3+.2,L=$$]{2}
   \Vertex[x=3 ,y=5+.2,L=$$]{3}
   \Vertex[x=5 ,y=5+.2,L=$$]{4}
   \tikzset{VertexStyle/.style = {shape= circle,
ball color =  black,inner sep= 0pt,outer sep= 0pt,minimum size=6 pt}}     
   \Vertex[x=1.8 ,y=1.8+.2,L=$$]{5}
   \Vertex[x=5 ,y=1.8+.2,L=$$]{6}
      \Vertex[x=6.2 ,y=3+.2,L=$$]{7}
   \Vertex[x=6.2 ,y=6.2+.2,L=$$]{8}
   \Vertex[x=1.8 ,y=6.2+.2,L=$$]{9}
    \Edge[style={ultra thick},color=red](1)(2)  \Edge[style={ultra thick},color=red](2)(4)  \Edge[style={ultra thick},color=red](3)(4)
      \Edge[style={ultra thick},color=red](1)(3) 
      \Edge(1)(5)\Edge(2)(6)\Edge(5)(6)\Edge(4)(6)
           \Edge(2)(7)\Edge(4)(8)\Edge(7)(8)\Edge(3)(9)\Edge(8)(9)\Edge(5)(9)
       \node  [color=blue] at (4,-.4) {$(B)$};
  \end{tikzpicture}
   \qquad\qquad
  \begin{tikzpicture}[scale=.7]
   \tikzset{VertexStyle/.style = {shape= circle,
ball color =  red,inner sep= 0pt,outer sep= 0pt,minimum size=6 pt}}     
\Vertex[x=2 ,y=2.5,L=$$]{1}
   \Vertex[x=5 ,y=2.5,L=$$]{2}
   \Vertex[x=2 ,y=5.5,L=$$]{3}
   \Vertex[x=5 ,y=5.5,L=$$]{4}
   \tikzset{VertexStyle/.style = {shape= circle,
ball color =  black,inner sep= 0pt,outer sep= 0pt,minimum size=6 pt}}     
   \Vertex[x=3.5 ,y=3.5,L=$$]{5}
   \Vertex[x=6.5 ,y=3.5,L=$$]{6}
      \Edge[style={ultra thick},color=red](1)(2)  \Edge[style={ultra thick},color=red](2)(4)  \Edge[style={ultra thick},color=red](3)(4)
    \Edge[style={ultra thick},color=red](1)(3)       
    \Edge(1)(5)\Edge(3)(5)
      \Edge(2)(6)\Edge(4)(6)
     \node  [color=blue] at (4,-.5) {$(C)$};
  \end{tikzpicture}

 \end{center}
\noindent
Monophobicity of these neighborhoods can be verified directly by case-by-case analysis, or by using either Theorems \ref{dim2theorem} or \ref{suffcondition}, to be proved later.

\begin{example}\label{nmpexamples} Examples of $2$-cubes $H$ with non-monophobic (monotolerant) neighborhoods $N_G(H)$:
\end{example}

\medskip

\begin{center}
\begin{tikzpicture}[scale=.7]
\tikzset{VertexStyle/.style = {shape= circle,
ball color = red,inner sep= 0pt,outer sep= 0pt,minimum size= 6 pt}}   
\Vertex[x=3 ,y=3,L=$$]{1}
   \Vertex[x=5 ,y=3,L=$$]{2}
   \Vertex[x=3 ,y=5,L=$$]{3}
   \Vertex[x=5 ,y=5,L=$$]{4}
   \tikzset{VertexStyle/.style = {shape= circle,
ball color =  black,inner sep= 0pt,outer sep= 0pt,minimum size=6 pt}}     
   \Vertex[x=4 ,y=1,L=$$]{5}
   \Vertex[x=7 ,y=4,L=$$]{6}
   \Vertex[x=4 ,y=7,L=$$]{7}
   \Vertex[x=1 ,y=4,L=$$]{8}
    \Edge[style={ultra thick},color=red](1)(2)  \Edge[style={ultra thick},color=red](2)(4)  \Edge[style={ultra thick},color=red](3)(4)
    \Edge[style={ultra thick},color=red](1)(3)       
    \Edge(1)(5)\Edge(2)(5)\Edge(8)(5)\Edge(5)(6)
      \Edge(2)(6)\Edge(4)(6)\Edge(6)(7)
      \Edge(4)(7)\Edge(3)(7)\Edge(7)(8)
      \Edge(3)(8)\Edge(1)(8)
  \node  [color=blue] at (4,-.5) {$(D)$};
  \end{tikzpicture}
 \qquad
 \begin{tikzpicture}[scale=.7]
\tikzset{VertexStyle/.style = {shape= circle,
ball color = red,inner sep= 0pt,outer sep= 0pt,minimum size=6 pt}}   
\Vertex[x=2 ,y=2,L=$$]{1}
   \Vertex[x=5 ,y=2,L=$$]{2}
   \Vertex[x=4 ,y=4,L=$$]{3}
   \Vertex[x=7 ,y=4,L=$$]{4}
   \tikzset{VertexStyle/.style = {shape= circle,
ball color =  black,inner sep= 0pt,outer sep= 0pt,minimum size=6 pt}}     
   \Vertex[x=4.5 ,y=7,L=$$]{5}
      \Edge[style={ultra thick},color=red](1)(2)  \Edge[style={ultra thick},color=red](2)(4)  \Edge[style={ultra thick},color=red](3)(4)
    \Edge[style={ultra thick},color=red](1)(3)       
   \Edge(2)(5)
      \Edge(3)(5)\Edge(4)(5)
     \node  [color=blue] at (4,0) {$(E)$};
  \end{tikzpicture}
   \qquad\qquad
  \begin{tikzpicture}[scale=.7]
\tikzset{VertexStyle/.style = {shape= circle,
ball color = red,inner sep= 0pt,outer sep= 0pt,minimum size= 7 pt}}   
\Vertex[x=2 ,y=2.5,L=$$]{1}
   \Vertex[x=5 ,y=2.5,L=$$]{2}
   \Vertex[x=2 ,y=5.5,L=$$]{3}
   \Vertex[x=5 ,y=5.5,L=$$]{4}
   \tikzset{VertexStyle/.style = {shape= circle,
ball color =  black,inner sep= 0pt,outer sep= 0pt,minimum size=6 pt}}     
   \Vertex[x=3.5 ,y=3.5,L=$$]{5}
   \Vertex[x=6.5 ,y=3.5,L=$$]{6}
      \Edge[style={ultra thick},color=red](1)(2)  \Edge[style={ultra thick},color=red](2)(4)  \Edge[style={ultra thick},color=red](3)(4)
    \Edge[style={ultra thick},color=red](1)(3)       
    \Edge(1)(5)\Edge(3)(5)
      \Edge(2)(6)\Edge(4)(6)\Edge(5)(6)
     \node  [color=blue] at (4,-.5) {$(F)$};
  \end{tikzpicture}

\end{center}

\noindent
A triangle affords the simplest example of a monotolerant neighborhood. For example, suppose that $N_G(H)$ is labeled as shown:
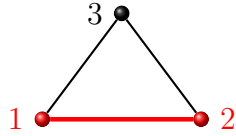
\begin{center}
 \begin{tikzpicture}[scale=.7]
\tikzset{VertexStyle/.style = {shape= circle,
ball color = red,inner sep= 0pt,outer sep= 0pt,minimum size= 6 pt}}   
\Vertex[LabelOut,Lpos=180,x=2 ,y=2,L=${\color{red}1}$]{1}
   \Vertex[LabelOut,Lpos=0,x=5 ,y=2,L=${\color{red}2}$]{2}
   \tikzset{VertexStyle/.style = {shape= circle,
ball color =  black,inner sep= 0pt,outer sep= 0pt,minimum size=6 pt}}     
   \Vertex[LabelOut,Lpos=180,x=3.5 ,y=4,L=$3$]{3}
      \Edge[style={ultra thick},color=red](1)(2)  \Edge[style={ultra thick},color=red](1)(2)  
   \Edge[style={thick},color=black](1)(3)  \Edge[style={thick},color=black](2)(3) 
    \end{tikzpicture}
  \end{center}
  
   \medskip\noindent
Then we can construct a singular 2-cube $\sigma=$
    \raisebox{-1.5em}{
\begin{tikzpicture}[scale=1]
\tikzset{VertexStyle/.style = {shape= circle,
ball color = gray,text= black,inner sep= 0pt,outer sep= 0pt,minimum size= 5pt}} 
   \Vertex[LabelOut,Lpos=180,x=0 ,y=0,L=${\color{red}1}$]{1}
     \Vertex[LabelOut,Lpos=0,x=1 ,y=0,L=${\color{red}2}$]{2}
       \Vertex[LabelOut,Lpos=180,x=0 ,y=1,L=${\color{red}3}$]{3}
         \Vertex[LabelOut,Lpos=0,x=1 ,y=1,L=${\color{red}3}$]{4}
    \Edge[style={thick},color=red](1)(2)
      \Edge[style={thick},color=blue](3)(4)
        \Edge[style={thick},color=blue](1)(3)
          \Edge[style={thick},color=blue](2)(4)
    \end{tikzpicture}
    }
    in which the base 1-cube $H = 12$ appears only once. In such cases we call $\sigma$ a {\em monotolerant witness} for $N_G(H)$.

\begin{definition}A rigid $d$-cube $H$ in a graph $G$ {\em has a grid neighborhood} if its $1$-neighborhood $N_G(H)$ can be embedded in the $1$-neighborhood of a $d$-cube in $I_\infty^d$.
\end{definition}
\noindent
Example 2.3(A) displays a $2$-cube with a grid neighborhood.  Example 2.3(B) plays a central role in Example \ref{ex3.3} below. 
In Example 2.4(D), the neighborhood is a $3$-cube with additional edges, yielding a monotolerant witness immediately. Note the small difference between Examples 2.3(C) and Example 2.4(F).

\begin{theorem}[First Main Theorem]\label{maintheorem}
If $G$ is a graph and $H$ is a rigid $d$-cube in $G$ with $H$ monophobic in $G$, then
\begin{enumerate}
\item 
If $\gamma \in C_d(G)$ is a homology cycle containing exactly one nonzero term supported by $H$, then $[\gamma]\ne [0]$ in $H_d(G)$.
\item If $\alpha: (I_\infty^d, \partial I_\infty^d) \to (G,v_0)$ is a based graph map in which small cubes supported by $H$ appear an odd number of times, then
$[\alpha] \ne [0]$ in $A_d(G,v_0)$.
\end{enumerate}
\end{theorem}

\begin{theorem} [Second Main Theorem]\label{gridlemma}
For all $d\ge 1$, if $H$ is a rigid $d$-cube in $G$ with a grid neighborhood, then $H$ is monophobic in $G$.
\end{theorem}

Situations in which Theorems \ref{maintheorem} and \ref{gridlemma} can be applied are abundant.  We illustrate  with three examples.

\begin{example}\label{ex3.3} {\bf 10-Vertex Crown.} The following graph appeared in \cite{BGJW1} and motivated much of the present work. 
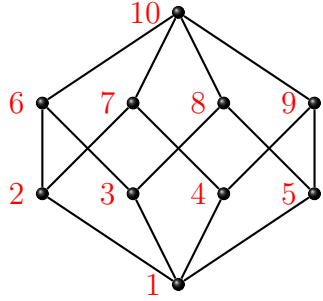
\begin{center}
\begin{tikzpicture}[scale=1.2]
\tikzset{VertexStyle/.style = {shape= circle,
ball color = black,text= black,inner sep= 0pt,outer sep= 0pt,minimum size= 5 pt}}

   \Vertex[LabelOut,Lpos=180,x=1.5,y=0,L=${\color{red}1}$]{1}
   \Vertex[LabelOut,Lpos=180,x=0 ,y=1,L=${\color{red}2}$]{2}
   \Vertex[LabelOut,Lpos=180,x=1,y=1,L=${\color{red}3}$]{3}
    \Vertex[LabelOut,Lpos=180,x=2,y=1,L=${\color{red}4}$]{4}
     \Vertex[LabelOut,Lpos=180,x=3,y=1,L=${\color{red}5}$]{5}
      \Vertex[LabelOut,Lpos=180,x=0,y=2,L=${\color{red}6}$]{6}
      \Vertex[LabelOut,Lpos=180,x=1,y=2,L=${\color{red}7}$]{7}
      \Vertex[LabelOut,Lpos=180,x=2,y=2,L=${\color{red}8}$]{8}
          \Vertex[LabelOut,Lpos=180,x=3,y=2,L=${\color{red}9}$]{9}
                 \Vertex[LabelOut,Lpos=180,x=1.5,y=3,L=${\color{red}10}$]{10}
    \Edge[style={thick},color=black](1)(2)
    \Edge[style={thick},color=black](1)(3)  \Edge[style={thick},color=black](1)(4)  \Edge[style={thick},color=black](1)(5)  
    \Edge[style={thick},color=black](2)(6) \Edge[style={thick},color=black](2)(7) \Edge[style={thick},color=black](3)(8)\Edge[style={thick},color=black](3)(6)
    \Edge[style={thick},color=black](4)(7) \Edge[style={thick},color=black](4)(9)\Edge[style={thick},color=black](5)(8)\Edge[style={thick},color=black](5)(9)
    \Edge[style={thick},color=black](6)(10)\Edge[style={thick},color=black](7)(10)\Edge[style={thick},color=black](8)(10)\Edge[style={thick},color=black](9)(10)
    \end{tikzpicture}
\end{center}
\end{example}
\noindent
There is an associated cell complex with eight quadrangular $2$-cells forming a ``natural" $2$-cycle (see Section 6 for more discussion).  One can construct a corresponding  $2$-cycle $\theta \in C_2(G)$
consisting of singular 2-cubes in $C_2(G)$: 
{\tiny
 \[
 \begin{array}{|c|c|}
 \hline
 3&6\\
 \hline
1&2\\
 \hline
 \end{array}\;+
  \begin{array}{|c|c|}
 \hline
 2&7\\
 \hline
 1&4\\
 \hline
\end{array}
\;+
  \begin{array}{|c|c|}
 \hline
 5&8\\
 \hline
 1&3\\
 \hline
\end{array}
\;+
  \begin{array}{|c|c|}
 \hline
 4&9\\
 \hline
 1&5\\
 \hline
\end{array}
\;+
 \begin{array}{|c|c|}
 \hline
 6&10\\
 \hline
 2&7\\
 \hline
\end{array}
\;+
 \begin{array}{|c|c|}
 \hline
 8&10\\
 \hline
 3&6\\
 \hline
\end{array}
\;+
 \begin{array}{|c|c|}
 \hline
 7&10\\
 \hline
 4&9\\
 \hline
\end{array}
\;+
 \begin{array}{|c|c|}
 \hline
9&10\\
 \hline
 5&8\\
 \hline
\end{array}
\]
}
None of the supporting cubes in $G$ has a grid neighborhood, but 
 as noted above in Example 2.3(B), each of these $2$-cubes is monophobic. Consequently, 
Theorem \ref{maintheorem} implies that $H_2(G)$ and $A_2(G,v_0)$ are nontrivial.

\begin{example} {\bf Boundary of Subdivided Cube.} \label{ex3.1} Let $G_m^d=\partial I_m^d, d>1, m \ge 3$, e.g. for $d=3, m=3$, we obtain the following diagram
(showing  front-visible vertices):  

\medskip
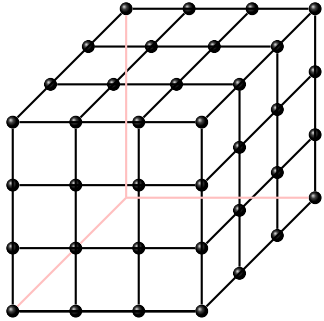
\begin{center}
\begin{tikzpicture}[scale=.5]
\tikzset{VertexStyle/.style = {shape= circle,
ball color = black,inner sep= 0pt,outer sep= 0pt,minimum size= 5pt}} 
\Vertex[x=0,y=0,L=$$]{1}
\Vertex[x=5/3,y=0,L=$$]{2}
\Vertex[x=10/3,y=0,L=$$]{3}
\Vertex[x=5,y=0,L=$$]{4}
\Vertex[x=6,y=1,L=$$]{5}
\Vertex[x=7,y=2,L=$$]{6}
\Vertex[x=8,y=3,L=$$]{7}
\Vertex[x=0,y=5/3,L=$$]{8}
\Vertex[x=5,y=5/3,L=$$]{9}
\Vertex[x=8,y=3+5/3,L=$$]{10}
\Vertex[x=0,y=10/3,L=$$]{11}
\Vertex[x=5,y=10/3,L=$$]{12}
\Vertex[x=8,y=3+10/3,L=$$]{13}
\Vertex[x=0,y=5,L=$$]{14}
\Vertex[x=5/3,y=5,L=$$]{15}
\Vertex[x=10/3,y=5,L=$$]{16}
\Vertex[x=5,y=5,L=$$]{17}
\Vertex[x=1,y=6,L=$$]{18}
\Vertex[x=6,y=6,L=$$]{19}
\Vertex[x=2,y=7,L=$$]{20}
\Vertex[x=7,y=7,L=$$]{21}
\Vertex[x=3,y=8,L=$$]{22}
\Vertex[x=5/3+3,y=8,L=$$]{23}
\Vertex[x=10/3+3,y=8,L=$$]{24}
\Vertex[x=8,y=8,L=$$]{25}
\Vertex[x=5/3,y=5/3,L=$$]{26}
\Vertex[x=5/3,y=10/3,L=$$]{27}
\Vertex[x=10/3,y=5/3,L=$$]{28}
\Vertex[x=10/3,y=10/3,L=$$]{29}
\Vertex[x=6,y=1+5/3,L=$$]{30}
\Vertex[x=6,y=1+10/3,L=$$]{29}
\Vertex[x=7,y=2+5/3,L=$$]{31}
\Vertex[x=7,y=2+10/3,L=$$]{32}
\Vertex[x=5/3+1,y=5+1,L=$$]{30}
\Vertex[x=5/3+2,y=5+2,L=$$]{29}
\Vertex[x=10/3+1,y=5+1,L=$$]{31}
\Vertex[x=10/3+2,y=5+2,L=$$]{32}
\tikzset{VertexStyle/.style = {shape= circle,
ball color = green,inner sep= 0pt,outer sep= 0pt,minimum size= 0pt}} 
\Vertex[x=3,y=3,L=$$]{99}
 \Edge[style={thick},color=black](1)(4)
  \Edge[style={thick},color=black](1)(4)
 \Edge[style={thick},color=black](1)(4)
\Edge[style={thick},color=black](8)(9)
 \Edge[style={thick},color=black](11)(12)
 \Edge[style={thick},color=black](1)(4)
 \Edge[style={thick},color=black](14)(17)
 \Edge[style={thick},color=black](18)(19)
  \Edge[style={thick},color=black](4)(17)
 \Edge[style={thick},color=black](22)(25)
 \Edge[style={thick},color=black](1)(14)
  \Edge[style={thick},color=black](2)(15)
 \Edge[style={thick},color=black](3)(16)
 \Edge[style={thick},color=black](5)(19)
 \Edge[style={thick},color=black](6)(21)
 \Edge[style={thick},color=black](7)(4)
 \Edge[style={thick},color=black](1)(4)
 \Edge[style={thick},color=black](14)(22) 
  \Edge[style={thick},color=black](17)(25)
   \Edge[style={thick},color=black](16)(24)
   \Edge[style={thick},color=black](15)(23)
    \Edge[style={thick},color=black](20)(21) 
    \Edge[style={thick},color=black](7)(25)
     \Edge[style={thick},color=black](12)(13) 
     \Edge[style={thick},color=black](9)(10)
        \Edge[style={thick},color=pink](22)(99)
           \Edge[style={thick},color=pink](1)(26)
             \Edge[style={thick},color=pink](26)(99)
             \Edge[style={thick},color=pink](7)(99)
\end{tikzpicture}
\end{center}
\medskip
\end{example}
\noindent As in the previous example, one can construct a ``natural'' cycle of singular $d$-cubes in $C_d(G_m^d) $, although it is not completely trivial to do so for general $d$ (see Section 6). In this cycle, there are evidently 
$d$-cells with a grid neighborhood, Hence Theorems \ref{maintheorem} and \ref{gridlemma} both apply,  proving that
$H_d(G_m^d) \ne (0)$ and $A_d(G_m^d,v_0) \ne (0)$. 
Note: although only the non-corner cells have grid neighborhoods, it can be shown (via Theorems \ref{dim2theorem} or \ref{suffcondition}) that all of the cells are
monophobic in $G_m^d$.

\begin{example} \label{ex3.2}
{\bf Box Suspension of a Cycle.} Let $G= (S_k^{\Box}) \ZZ_m,  k\ge 5, m \ge 5$, where $S^\Box_k G$ denotes the graph obtained from $G$ by taking  the product $Z_m \Box I_k$ and identifying the first and last component to a single point, and $\ZZ_m$ denotes a $m$-cycle. The following diagram illustrates  $S_5^\Box \ZZ_5$:

\medskip
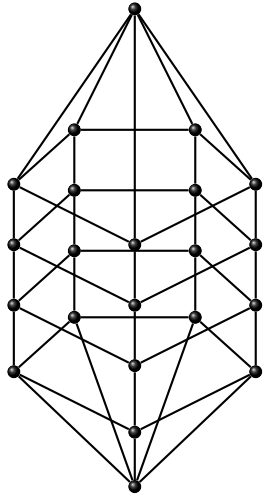
\begin{center}
\begin{tikzpicture}[scale=.8]
\tikzset{VertexStyle/.style = {shape= circle,
ball color = black,inner sep= 0pt,outer sep= 0pt,minimum size= 5pt}} 
\Vertex[x=3,y=0,L=$$]{1}
\Vertex[x=5,y=1,L=$$]{2}
\Vertex[x=4,y=1.9,L=$$]{3}
\Vertex[x=2,y=1.9,L=$$]{4}
\Vertex[x=1,y=1,L=$$]{5}
\Vertex[x=3,y=0+1.1,L=$$]{6}
\Vertex[x=5,y=1+1.1,L=$$]{7}
\Vertex[x=4,y=1.9+1.1,L=$$]{8}
\Vertex[x=2,y=1.9+1.1,L=$$]{9}
\Vertex[x=1,y=1+1.1,L=$$]{10}
\Vertex[x=3,y=0+2.1,L=$$]{11}
\Vertex[x=5,y=1+2.1,L=$$]{12}
\Vertex[x=4,y=1.9+2.1,L=$$]{13}
\Vertex[x=2,y=1.9+2.1,L=$$]{14}
\Vertex[x=1,y=1+2.1,L=$$]{15}
\Vertex[x=3,y=0+3.1,L=$$]{16}
\Vertex[x=5,y=1+3.1,L=$$]{17}
\Vertex[x=4,y=1.9+3.1,L=$$]{18}
\Vertex[x=2,y=1.9+3.1,L=$$]{19}
\Vertex[x=1,y=1+3.1,L=$$]{20}
\Vertex[x=3,y=-.9,L=$$]{0}
\Vertex[x=3,y=7,L=$$]{21}
 \Edges[style={thick},color=black](1,2,3,4,5,1)
  \Edges[style={thick},color=black](6,7,8,9,10,6)
   \Edges[style={thick},color=black](11,12,13,14,15,11)
     \Edges[style={thick},color=black](16,17,18,19,20,16)
 \Edges[style={thick},color=black](1,6,11,16)
 \Edges[style={thick},color=black](2,7,12,17)
 \Edges[style={thick},color=black](3,8,13,18)
 \Edges[style={thick},color=black](4,9,14,19)
 \Edges[style={thick},color=black](5,10,15,20)
 \foreach \v in {1,2,3,4,5}{\Edge[style={thick},color=black](0)(\v)}
  \foreach \v in {16,17,18,19, 20}{\Edge[style={thick},color=black](21)(\v)}
 \end{tikzpicture}
 \end{center}
\medskip
\end{example}
\noindent 
As in the previous two examples, there is a cell complex with a natural $2$-cycle, and some of its $2$-cubes have 
a grid neighborhood; in fact, it can be shown (via Theorems \ref{dim2theorem} or \ref{suffcondition}) that all of the 
quadrilateral cells are monophobic,
and it is not difficult to construct a corresponding cycle in $C_2(G)$. Hence Theorems \ref{maintheorem} and 
\ref{gridlemma} apply and we obtain that 
$H_2(G) \ne (0)$.
Similarly, there are natural based graph maps $\alpha: (I_\infty^2, \partial I_\infty^2) \to (G,v_0)$ in which 
each small square appears exactly once, implying that $A_2(G,v_0) \ne (0)$. 

Examples (\ref{ex3.1}) and (\ref{ex3.2}) illustrate situations that generalize readily to higher dimensions, producing local grid 
neighborhoods. In
such cases Theorem \ref{gridlemma} applies. Example (\ref{ex3.3}) is an isolated case with no grid neighborhoods, and suggests 
searching for  more examples of this type, where Theorem \ref{gridlemma} would apply.  We state this officially as a problem:

\begin{problem}\label{mainproblem}
Find the ``right'' graphical condition (weaker than ``grid neighborhood'') equivalent to the property that a $d$-cube in $G$ is monophobic, for all $d\ge 1$.
\end{problem}

In Section 4 of this paper we solve this problem in dimension 2, characterizing monophobic neighborhoods of a $2$-cube
by means of short list of excluded configurations. In Section 5 we obtain another condition for monophobicity that is necessary 
and sufficient in dimension 2 and sufficient in dimensions $d>2$.

\section{Proofs of Lemmas and Theorems}
We begin by establishing some notation and conventions that will be used frequently.
The {\em standard $d$-cube} $Q_d$ is defined to be the set of $d$-tuples $\{(\epsilon_1,\epsilon_2,\dots,\epsilon_d )\;|\; \epsilon_i=\pm 1\}$. 
The {\em hyperoctahedral group} $\HH_d$ acts on $Q_d$ as left multiplication by {\em signed permutation matrices}. This action also provides an
action of $\HH_d$ on {\em singular $d$-cubes} $\sigma: Q_d \to G$ by $P\sigma (v) = \sigma(P^{-1} v)$, where $v$ is a vertex of $Q_d$.

The hyperoctahedral group $\HH_d$ may also be defined as the wreath product $\ZZ_2^d \rtimes S_d$, where $S_d$ denotes the symmetric group of order $d$. This leads to a canonical
representation of elements $P\in \HH_d$ as $P =P_\phi P_\theta$, where $P_\theta$ denotes a permutation of axes and $P_\phi$ is a product of
axis reflections. This representation is easily converted into {\em signed permutation} notation, e.g., $3\overline{1}4 \overline{5}2$ is interpreted as ``permute the axes according to the permutation $31452$, then negate the axes in positions $2$ and $4$". As is well known, every permutation in $S_d$ may be represented as a product of transpositions, and thus we have for example,
\[
P_1 P_3 P_{13} = (\overline{1}2 \overline{3} )(321)= \overline{3}2 \overline{1} = 
\begin{pmatrix}
 -1 & 0 & 0  \\
 0 & 1 & 0  \\
   0 & 0 & -1  
 \end{pmatrix}
\begin{pmatrix}
 0 & 0 & 1  \\
 0 & 1 & 0  \\
   1 & 0 & 0  
 \end{pmatrix}=
\begin{pmatrix}
 0 & 0 & -1  \\
0& 1 & 0  \\
  -1 & 0 &0  
 \end{pmatrix}
\]
The {\em sign} of an element $P\in \HH_d$ is defined as $(-1)^{\det(P)}$. If $P = P_\phi P_\theta$ is represented in wreath product
form as above, then
\[
\sgn(P) = (-1)^{\flip(\phi) + \inv(\theta)}
\]
where $\flip(\phi)$ denotes the number of axis reflections in $\phi$ and $\inv(\theta)$ denotes the number of inversions in $\theta$.

The $\epsilon_i$th face of $Q_d$, denoted $F_i^{\epsilon_i} Q_d$, is the set of $d$-tuples in $Q_d$ with $i$th component equal to $\epsilon_i$, and may be identified with $Q_{d-1}$ using the inherited ordering of indices.
If $\sigma: Q_d\to G$ is a singular $d$-cube, the $\epsilon_i$th face 
 $f_i^{\epsilon_i} \sigma: Q_{d-1}\to G$  is obtained by restricting $\sigma$ to $F_i^{\epsilon_i} Q_d$.  These operators
satisfy the  {\em face relations}  
(see e.g., \cite{BGJW3}) 
\[
f^\epsilon_i f^\theta_j = f^\theta_{j-1} f^\epsilon_i , \;\; i<j.
\]
The {\em support} of a face $f_i^{\epsilon_i} \sigma$ is its image $\img f_i^{\epsilon_i} \sigma \subseteq G$. 

\begin{lemma}\label{lemma0124}
For any $d \ge 1$, if $H$  is a rigid $d$-cube in $G$ and $\sigma$ is a singular $(d+1)$-cube, then $\sigma$ has either $0$, $1$, $2$, or $4$ faces supported by $H$. If it has $2$ faces, they are either 
adjacent or opposite (this is true of {\em any} pair of faces). If it has $4$ faces, they ``wrap around" the cube in a cycle of length $4$ containing two opposite pairs.
\end{lemma}
\begin{proof}
Let $H$ be an arbitrary rigid $d$-cube in $G$. 
Suppose that $\sigma$ is a singular $(d+1)$-cube, and that the number of $d$-faces supported by $H$ is $>1$. Suppose that
$f_i^a\sigma$ and $f_j^b \sigma$ are two such faces, such that $i<j$. In particular, they are not opposite faces.  We claim that the
only possible $d$-faces that can be supported by $H$ are 
\begin{equation}\label{onlypossible}
f_i^-\sigma, f_i^+\sigma, f_j^-\sigma, f_j^+\sigma.
\end{equation}
We will show that any other face, say $f_k^c\sigma$, $k \ne i,j$, must have a repeated label; in fact we will show that it has a repeated
$(d-2)$-face.
First note that $f_i^a\sigma$ and $f_j^b\sigma$ have a common $(d-1)$-face,
\[
f_i^a f_j^b \sigma = f_{j-1}^b f_i^a \sigma
\]
by the face relations. Rigidity of $H$ implies that the opposite faces in $f_j^b \sigma$ and $f_i^a\sigma$ must also agree, 
so we also have
\[
f_i^{-a}f_j^{b} \sigma = f_{j-1}^{-b} f_i^a \sigma.
\]
There are now several cases, depending on the relative order of $i,j,k$. Consider, for example, $i<k<j$ (the other cases are
similar). Applying $f_{k-1}^c$ to both sides of the last equation gives 
\[
f_{k-1}^c f_i^{-a}f_j^{b} \sigma =f_{k-1}^c f_{j-1}^{-b} f_i^a \sigma,
\]
which reduces by the face relations to
\[
f_{j-2}^{b} f_{i}^{-a} f_k^c \sigma = f_{j-2}^{-b} f_i^{a} f_k^c \sigma.
\]
This shows that $f_k^c \sigma$ has two disjoint $(d-1)$-faces ($f_{i}^{-a} f_k^c \sigma$ and $f_i^{a} f_k^c \sigma$) containing identical
subfaces, and thus $f_k^c \sigma$ is not supported by $H$.

We have shown that only faces listed in (\ref{onlypossible}) can be supported by $H$. An easy argument shows that
if three of them are supported by $H$, then so is the fourth. This completes the proof of Lemma \ref{lemma0124}.
\end{proof}

If $H$ is monophobic in $G$,  it follows from Lemma \ref{lemma0124} that a singular $(d+1)$-cube $\sigma$ must 
contain only 0, 2, or 4 faces
supported by $H$.  These faces of $\sigma$ may differ by (uniquely determined) symmetries in $\HH_{d}$. To keep track of 
the symmetries, it will be convenient to 
designate
a base singular cube $\beta_0^H$, chosen arbitrarily for each rigid $d$-cube in $G$.

For notational convenience, we will sometimes identify $Q^{d}$ and $Q^{d+1}$ with the 
cubes $Q^0_d =\{0,1\}^{d}$ and  $Q^0_{d+1} = \{ 0,1\}^{d+1}$, respectively.
The symbols  $\vec{0}^{d} $  and  $\vec{0}^{d+1} $  denote all-zeros vectors, 
and $\vec{e}^{\,d}_i$ and $\vec{e}^{\,d+1}_i$ denote the $i$th coordinate vectors.
 More generally, for a set of indices $S$, 
let $\vec{e}^{\,d}_S$ denote $\sum_{i\in S} \vec{e}^{\,d}_i$, with $\vec{e}^{\,d+1}_S$ defined similarly.

The {\em boundary}
of a singular $d$-cube $\sigma$ is defined as
\[
\partial \sigma = \sum_{i=1}^d (-1)^i (f_i^{-} - f_i^{+}) \sigma = \sum_{i, \epsilon_i} (-1)^{i-1}\epsilon_i f_i^{\epsilon_i}\sigma,
\]
where $\epsilon_i = \pm 1$.

The next lemma provides a key step in the proof of Theorem \ref{maintheorem}, defining a linear invariant
$\chi$ on $d$-faces that is guaranteed to vanish on boundaries.

\begin{lemma}[Main Lemma]\label{chilemma}
If $H$ is a monophobic $d$-cube in $G$ and $\sigma\in C_{d+1}(G)$ is a singular $(d+1)$-cube,  let ${\mathcal F}(\sigma)$ denote the set of faces of $\sigma$. Define a function $\chi: {\mathcal F}(\sigma) \to \{0, \pm 1\}$ 
by 
\[
\chi(f_i^{\epsilon_i}\sigma) = 
\begin{cases}
0 & \textrm{if $f_i^{\epsilon_i}\sigma$ is not supported by $H$}, \\
\mathrm{sign} (P) & \textrm{if $f_i^{\epsilon_i}\sigma= P\beta_0^H, P \in \HH_d$}. \\
\end{cases}
\]
Then $\chi(\partial \sigma) = 0$.
\end{lemma}

\noindent
The following example shows a typical computation of $\chi $ as in Lemma \ref{chilemma}. Here  $\beta_0^H = f_3^- \sigma$.

\begin{example}
Let
$\beta_0^H=$
\raisebox{-1.7em}{
\begin{tikzpicture}[scale=1]
\tikzset{VertexStyle/.style = {shape= circle,
ball color = red,inner sep= 0pt,outer sep= 0pt,minimum size=6pt}} 
   \Vertex[LabelOut,Lpos=180,x=0 ,y=0,L=${\color{red}1}$]{1}
     \Vertex[LabelOut,Lpos=0,x=1 ,y=0,L=${\color{red}2}$]{2}
       \Vertex[LabelOut,Lpos=180,x=0 ,y=1,L=${\color{red}3}$]{3}
         \Vertex[LabelOut,Lpos=0,x=1 ,y=1,L=${\color{red}4}$]{4}
    \Edge[style={thick},color=red](1)(2)
      \Edge[style={ thick},color=red](3)(4)
        \Edge[style={thick},color=red](1)(3)
          \Edge[style={thick},color=red](2)(4)
    \end{tikzpicture}
},
and let $\sigma = $
    \raisebox{-3.3em}{
\begin{tikzpicture}[scale=1.5]
\tikzset{VertexStyle/.style = {shape= circle,
ball color = gray,text= black,inner sep= 0pt,outer sep= 0pt,minimum size= 5pt}} 
   \Vertex[LabelOut,Lpos=180,x=0 ,y=0,L=${\color{red}1}$]{1}
     \Vertex[LabelOut,Lpos=0,x=1 ,y=0,L=${\color{red}2}$]{2}
       \Vertex[LabelOut,Lpos=180,x=.5 ,y=.5,L=${\color{red}3}$]{3}
         \Vertex[LabelOut,Lpos=0,x=1.5 ,y=.5,L=${\color{red}4}$]{4}
            \Vertex[LabelOut,Lpos=180,x=0 ,y=1,L=${\color{red}2}$]{5}
               \Vertex[LabelOut,Lpos=0,x=1 ,y=1,L=${\color{red}1}$]{6}
                  \Vertex[LabelOut,Lpos=180,x=.5 ,y=1.5,L=${\color{red}4}$]{7}
                     \Vertex[LabelOut,Lpos=0,x=1.5 ,y=1.5,L=${\color{red}3}$]{8}
    \Edge[style={thick},color=blue](1)(2)
      \Edge[style={thick},color=blue](3)(4)
        \Edge[style={thick},color=blue](1)(3)
          \Edge[style={thick},color=blue](2)(4)
           \Edge[style={thick},color=blue](5)(6)
      \Edge[style={thick},color=blue](7)(8)
        \Edge[style={thick},color=blue](5)(7)
          \Edge[style={thick},color=blue](6)(8)
          \Edge[style={thick},color=blue](1)(5)
      \Edge[style={thick},color=blue](2)(6)
        \Edge[style={thick},color=blue](3)(7)
          \Edge[style={ thick},color=blue](4)(8)
    \end{tikzpicture}
    }.
   There are $4$ faces supported by $H$, wrapping around the cube:
    
    \bigskip
    
 \qquad   
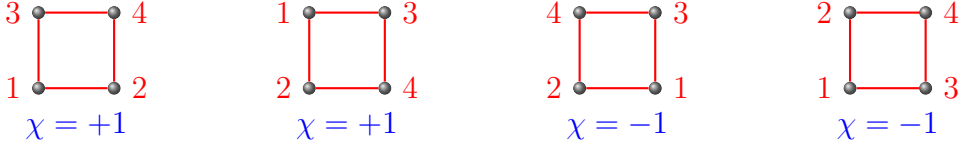
\begin{tikzpicture}[scale=1]
\tikzset{VertexStyle/.style = {shape= circle,
ball color = gray,inner sep= 0pt,outer sep= 0pt,minimum size= 5pt}} 
   \Vertex[LabelOut,Lpos=180,x=0 ,y=0,L=${\color{red}1}$]{1}
     \Vertex[LabelOut,Lpos=0,x=1 ,y=0,L=${\color{red}2}$]{2}
       \Vertex[LabelOut,Lpos=180,x=0 ,y=1,L=${\color{red}3}$]{3}
         \Vertex[LabelOut,Lpos=0,x=1 ,y=1,L=${\color{red}4}$]{4}
    \Edge[style={thick},color=red](1)(2)
      \Edge[style={thick},color=red](3)(4)
        \Edge[style={thick},color=red](1)(3)
          \Edge[style={thick},color=red](2)(4)
            \node  [color=blue] at (.5,-.5) {$\chi=+1$};
    \end{tikzpicture}
    \qquad\quad
    \begin{tikzpicture}[scale=1]
\tikzset{VertexStyle/.style = {shape= circle,
ball color = gray,inner sep= 0pt,outer sep= 0pt,minimum size= 5pt}} 
   \Vertex[LabelOut,Lpos=180,x=0 ,y=0,L=${\color{red}2}$]{1}
     \Vertex[LabelOut,Lpos=0,x=1 ,y=0,L=${\color{red}4}$]{2}
       \Vertex[LabelOut,Lpos=180,x=0 ,y=1,L=${\color{red}1}$]{3}
         \Vertex[LabelOut,Lpos=0,x=1 ,y=1,L=${\color{red}3}$]{4}
    \Edge[style={thick},color=red](1)(2)
      \Edge[style={thick},color=red](3)(4)
        \Edge[style={thick},color=red](1)(3)
          \Edge[style={thick},color=red](2)(4)
            \node  [color=blue] at (.5,-.5) {$\chi=+1$};
    \end{tikzpicture}
    \qquad\quad
     \begin{tikzpicture}[scale=1]
\tikzset{VertexStyle/.style = {shape= circle,
ball color = gray,inner sep= 0pt,outer sep= 0pt,minimum size= 5pt}} 
   \Vertex[LabelOut,Lpos=180,x=0 ,y=0,L=${\color{red}2}$]{1}
     \Vertex[LabelOut,Lpos=0,x=1 ,y=0,L=${\color{red}1}$]{2}
       \Vertex[LabelOut,Lpos=180,x=0 ,y=1,L=${\color{red}4}$]{3}
         \Vertex[LabelOut,Lpos=0,x=1 ,y=1,L=${\color{red}3}$]{4}
    \Edge[style={thick},color=red](1)(2)
      \Edge[style={thick},color=red](3)(4)
        \Edge[style={thick},color=red](1)(3)
          \Edge[style={thick},color=red](2)(4)
            \node  [color=blue] at (.5,-.5) {$\chi=-1$};
    \end{tikzpicture}
    \qquad\quad
     \begin{tikzpicture}[scale=1]
\tikzset{VertexStyle/.style = {shape= circle,
ball color = gray,inner sep= 0pt,outer sep= 0pt,minimum size= 5pt}} 
   \Vertex[LabelOut,Lpos=180,x=0 ,y=0,L=${\color{red}1}$]{1}
     \Vertex[LabelOut,Lpos=0,x=1 ,y=0,L=${\color{red}3}$]{2}
       \Vertex[LabelOut,Lpos=180,x=0 ,y=1,L=${\color{red}2}$]{3}
         \Vertex[LabelOut,Lpos=0,x=1 ,y=1,L=${\color{red}4}$]{4}
    \Edge[style={thick},color=red](1)(2)
      \Edge[style={thick},color=red](3)(4)
        \Edge[style={thick},color=red](1)(3)
          \Edge[style={thick},color=red](2)(4)
            \node  [color=blue] at (.5,-.5) {$\chi=-1$};
    \end{tikzpicture}
    \smallskip
    \newline
   It follows from these data that $\chi(\partial \sigma) = 0$, as asserted by Lemma \ref{chilemma}.

\end{example}

The proof of Lemma \ref{chilemma} will involve  understanding carefully how the signs of terms in $\partial$ interact with signs of symmetries $P\in \HH_d$.   Some additional notation is needed for this analysis.

\begin{definition}[Signs of terms in $\partial$] Let $F_i^{\epsilon_i}, F_j^{\epsilon_j}$ be faces of $Q_d$, and let $\sigma: Q_d \to G$ be a singular $d$-cube. Define
\begin{align*}
 \sgn (F_i^{\epsilon_i}) &= \textrm{sign associated with $f_i^{\epsilon_i}\sigma$  in  $\partial \sigma$} \\
\sgn (F_i^{\epsilon_i},F_j^{\epsilon_j})  &= \textrm{sign change from $ \sgn (F_i^{\epsilon_i})$  to $\sgn (F_j^{\epsilon_j})$ } 
\end{align*}
Specifically:
\[
\sgn (F_i^{\epsilon_i}) = (-1)^{i-1}\epsilon_i\;\; \mbox{and}\;\;
\sgn (F_i^{\epsilon_i},F_j^{\epsilon_j}) = (-1)^{i+j} \epsilon_i \epsilon_j.
\]
\end{definition} 

\begin{definition}[Sign differences of adjacent faces]
Suppose that  $F_i^{\epsilon_i}$ and $F_j^{\epsilon_j}$ are adjacent faces of $Q_d$, and let 
$P:  Q_d \to Q_d$ be
the unique reflection of $Q_d$ fixing their intersection. Let $\tilde{P}: Q_{d-1} \to Q_{d-1}$ be the map obtained
by restricting $P$ to  $F_i^{\epsilon_i}$,  with ordering and orientations of axes inherited from $Q_d$.
Then define
\[
\chi(F_i^{\epsilon_i}, F_j^{\epsilon_j}) = \sgn(\tilde{P}).
\]
\end{definition}
\noindent
Note that while $\sgn(P)$ is always equal to $-1$, $\tilde{P}$ can have sign $\pm 1$.  Here are some illustrations of 
the computations for $d=3$.
\begin{align*}
F_3^- \longleftrightarrow F_1^+ \qquad & \chi(F_3^- , F_1^+) = +1 \\
F_3^- \longleftrightarrow F_2^+ \qquad & \chi(F_3^- , F_2^+) = -1 \\
F_3^- \longleftrightarrow F_1^- \qquad & \chi(F_3^- , F_1^-) = -1 \\
F_3^- \longleftrightarrow F_2^- \qquad & \chi(F_3^- , F_2^-) = +1 
\end{align*}
For example, $P = P_1 P_3 P_{13}: Q^0_3 \to Q^0_3$ reflects $F_1^+$ to $  F_3^-$  and induces the map $\tilde{P}= P_1 P_{12}: Q^0_2 \to Q^0_2$ (illustrated below), which is a rotation. Hence  $\chi(F_3^- , F_1^+) = \chi( F_1^+ ,F_3^- ) = +1$.

\medskip
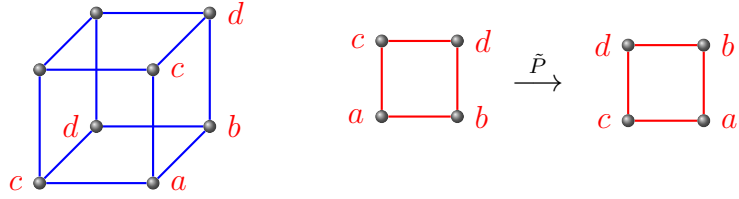
\qquad\qquad\quad
\raisebox{-2em}{
\begin{tikzpicture}[scale=1.5]
\tikzset{VertexStyle/.style = {shape= circle,
ball color = gray,text= black,inner sep= 0pt,outer sep= 0pt,minimum size= 5pt}} 
   \Vertex[LabelOut,Lpos=180,x=0 ,y=0,L=${\color{red}c}$]{1}
     \Vertex[LabelOut,Lpos=0,x=1 ,y=0,L=${\color{red}a}$]{2}
       \Vertex[LabelOut,Lpos=180,x=.5 ,y=.5,L=${\color{red}d}$]{3}
         \Vertex[LabelOut,Lpos=0,x=1.5 ,y=.5,L=${\color{red}b}$]{4}
            \Vertex[LabelOut,Lpos=180,x=0 ,y=1,L=${\color{red}}$]{5}
               \Vertex[LabelOut,Lpos=0,x=1 ,y=1,L=${\color{red}c}$]{6}
                  \Vertex[LabelOut,Lpos=180,x=.5 ,y=1.5,L=${\color{red}}$]{7}
                     \Vertex[LabelOut,Lpos=0,x=1.5 ,y=1.5,L=${\color{red}d}$]{8}
    \Edge[style={thick},color=blue](1)(2)
      \Edge[style={thick},color=blue](3)(4)
        \Edge[style={thick},color=blue](1)(3)
          \Edge[style={thick},color=blue](2)(4)
           \Edge[style={thick},color=blue](5)(6)
      \Edge[style={thick},color=blue](7)(8)
        \Edge[style={thick},color=blue](5)(7)
          \Edge[style={thick},color=blue](6)(8)
          \Edge[style={thick},color=blue](1)(5)
      \Edge[style={thick},color=blue](2)(6)
        \Edge[style={thick},color=blue](3)(7)
          \Edge[style={ thick},color=blue](4)(8)
    \end{tikzpicture}
    }
   \qquad 
\begin{tikzpicture}[scale=1]
\tikzset{VertexStyle/.style = {shape= circle,
ball color = gray,inner sep= 0pt,outer sep= 0pt,minimum size= 5pt}} 
   \Vertex[LabelOut,Lpos=180,x=0 ,y=0,L=${\color{red}a}$]{1}
     \Vertex[LabelOut,Lpos=0,x=1 ,y=0,L=${\color{red}b}$]{2}
       \Vertex[LabelOut,Lpos=180,x=0 ,y=1,L=${\color{red}c}$]{3}
         \Vertex[LabelOut,Lpos=0,x=1 ,y=1,L=${\color{red}d}$]{4}
    \Edge[style={thick},color=red](1)(2)
      \Edge[style={thick},color=red](3)(4)
        \Edge[style={thick},color=red](1)(3)
          \Edge[style={thick},color=red](2)(4)
           \end{tikzpicture}
          \raisebox{1.5em}{$\overset{\tilde{P}} \longrightarrow$ }
    \begin{tikzpicture}[scale=1]
\tikzset{VertexStyle/.style = {shape= circle,
ball color = gray,inner sep= 0pt,outer sep= 0pt,minimum size= 5pt}} 
   \Vertex[LabelOut,Lpos=180,x=0 ,y=0,L=${\color{red}c}$]{1}
     \Vertex[LabelOut,Lpos=0,x=1 ,y=0,L=${\color{red}a}$]{2}
       \Vertex[LabelOut,Lpos=180,x=0 ,y=1,L=${\color{red}d}$]{3}
         \Vertex[LabelOut,Lpos=0,x=1 ,y=1,L=${\color{red}b}$]{4}
    \Edge[style={thick},color=red](1)(2)
      \Edge[style={thick},color=red](3)(4)
        \Edge[style={thick},color=red](1)(3)
          \Edge[style={thick},color=red](2)(4)
    \end{tikzpicture}
\bigskip

Returning to the proof of Lemma \ref{chilemma}, we suppose that a monophobic cube $H\subseteq G$ is supported by 
two faces in $\sigma$, which may be either adjacent or opposite. The proof of Lemma \ref{chilemma} consists of two
lemmas, dealing separately with these two cases.

\begin{lemma}\label{adjfaces}{\rm (Adjacent Faces)}
Let $F_i^{\epsilon_i}$ and $F_j^{\epsilon_j} $ be adjacent faces of $Q_d$ (i.e. $i \ne j$). Then
\[
\sgn(F_i^{\epsilon_i}, F_j^{\epsilon_j}) \chi(F_i^{\epsilon_i}, F_j^{\epsilon_j})=-1.
\]
Hence if two adjacent faces of $\sigma$ are supported by $H$, their contributions cancel out in the
computation of $\chi(\partial\sigma)$.
\end{lemma}

\begin{proof}
It is immediate that 
$F_i^{\epsilon_i}$ and $F_j^{\epsilon_j} $ are interchanged either by the reflection
$P=P_{ij}$ (if $\epsilon_i =\epsilon_j$) or by the reflection $P=P_iP_j P_{ij}$  (if $\epsilon_i \ne \epsilon_j$).
We know that 
\begin{equation}\label{sgnchi}
\sgn(F_i^{\epsilon_i},F_j^{\epsilon_j})= (-1)^{i+j}\epsilon_i \epsilon_j.
\end{equation}
 To compute
$\chi(F_i^{\epsilon_i},F_j^{\epsilon_j})$ we must identify $\tilde{P}$ and compute its sign. 

Without loss of generality, assume that $i<j$.
We may construct $\tilde{P}$ from $P$ as follows: 
first write $P$ as a signed permutation, then drop $i$ (or $\bar{i}$), as appropriate, and finally, reduce 
all occurrences of $j$ (or $\bar{j}$) by one if $j>i$.

For example, if $n=3$ and $P: F_1^+\longleftrightarrow F_3^-$ (as illustrated above) then $P = P_1P_3 P_{13}$, and
we get 
\[
P = \bar{3}2\bar{1}  \longrightarrow  \bar{3} 2 \longrightarrow  \bar{2} 1 = \tilde{P}
\]
It is clear from this construction that $(-1)^\inv(\tilde{P})= (-1)^{j-i-1}$ and 
$\flip(\tilde{P}) = 0$ or $1$ according to whether $\epsilon_i = \epsilon_j$ or $ \epsilon_i = -\epsilon_j$. Hence
\[
\chi(F_i^{\epsilon_i},F_j^{\epsilon_j} )= \sgn(\tilde{P}) = 
\begin{cases}
(-1)^{j-i-1} & \textrm{if } \epsilon_i = \epsilon_j \\
(-1)^{j-i} & \textrm{if } \epsilon_i = - \epsilon_j \\
\end{cases}
\]
which can be rewritten as 
\begin{equation}\label{sgnchi2}
\chi(F_i^{\epsilon_i},F_j^{\epsilon_j} )= (-1)^{j-i-1} \epsilon_i \epsilon_j. 
\end{equation}
Combining (\ref{sgnchi}) with (\ref{sgnchi2}) proves Lemma \ref{adjfaces}
\end{proof}

We will need the following purely graph-theoretic result, which may be of independent interest.

\begin{proposition}\label{pert}
Let $T: Q^0_d\to Q^0_d$ be a bijective graph map such that 
$d(v,Tv) \le 1$ for all vertices $v\in Q^0_d$, where $d$ denotes Hamming distance.  Then $T$ is either the identity,  a quarter turn $P_i P_{ij}$ 
involving two axis directions, or an axis reflection $P_i$ across a hyperplane separating a face from its opposite. 
\end{proposition}

\begin{proof} Suppose that $T$ is not the identity. Without loss of generality, we may assume that $T(\vec{0}) = \vec{e}_1$. 

\smallskip
\noindent
{\em Case 1:}  Suppose $T( \vec{e}^{\,d}_1)= \vec{0}^{d}$, i.e., $T$ switches $\vec{0}^{d}$ and $\vec{e}_1^{\,d}$. If $j \ne 1$, then 
$T(\vec{e}_j^{\,d})$ is a common neighbor of both $\vec{e}_j^{\,d}$ and $\vec{e}_1^{\,d}$, since $T$ is a graph map. This forces 
 $T(\vec{e}_j^{\,d})= \vec{0}$ or
$ \vec{e}_{1j}^{\,d}$.
Since $T$ is bijective,  it follows that $T(\vec{e}_j^{\,d})= \vec{e}_{1j}^{\,d}$, and  similar
reasoning shows that  $T( \vec{e}_{1j}^{\,d})=T(\vec{e}_j^{\,d})$. The argument extends easily by induction, showing that $T$ switches
$ \vec{e}_S^{\,d}$ and $ \vec{e}_S^{\,d}+\vec{e}_1^{\,d}$ for all $S \subseteq \{2,\dots,d\}$. Thus $T = P_1$, an axis reflection.

\smallskip
\noindent
{\em Case 2:} Suppose $T( \vec{e}^{\,d}_1)= \vec{e}^{\,d}_1+\vec{e}^{\,d}_j= \vec{e}^{\,d}_{1j}$ for some $j \ne 1$. Since $\vec{e}^{\,d}_j$ is adjacent to $\vec{0}^d$, 
$T(\vec{e}^{\,d}_j)$  is a common neighbor of  $T(\vec{0}^d) = \vec{e}^{\,d}_1$ and   $\vec{e}^{\,d}_j$. Bijectivity of $T$ implies 
$T(\vec{e}^{\,d}_j)\ne \vec{e}^{\,d}_{1j}$,  and hence  $T(\vec{e}^{\,d}_j) = \vec{0}^d$. A similar argument shows 
that $T(\vec{e}^{\,d}_{1j})= \vec{e}^{\,d}_j$,  thus establishing the $4$-cycle shown on the left below, with $T$ represented 
by red arrows:

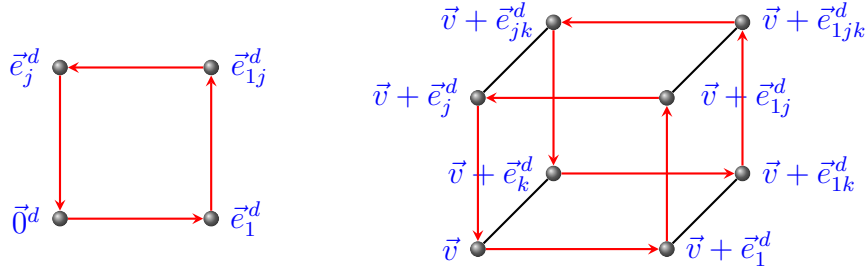
\begin{center}
\raisebox{.4cm}{
  \begin{tikzpicture}[scale=1]
\tikzset{VertexStyle/.style = {shape= circle,
ball color = gray,text= black,inner sep= 0pt,outer sep= 0pt,minimum size= 6pt}} 
   \Vertex[LabelOut,Lpos=180,Ldist = .01cm,x=0 ,y=0,L=$\color{blue}\vec{0}^d$]{1}
     \Vertex[LabelOut,Lpos=0,Ldist =.01,x=2 ,y=0,L=$\color{blue}\vec{e}^{\,d}_1$]{2}
       \Vertex[LabelOut,Lpos=0,Ldist = .01cm,x=2 ,y=2,L=$\color{blue}\vec{e}^{\,d}_{1j}$]{3}
         \Vertex[LabelOut,Lpos=180,Ldist = .01cm,x=0 ,y=2,L=$\color{blue}\vec{e}^{\,d}_j$]{4}
                      \Edges[style={-{stealth[length=4mm]},thick},color=red](1,2,3,4,1) 
    \end{tikzpicture}
    }
    \qquad
 \begin{tikzpicture}[scale=1]
\tikzset{VertexStyle/.style = {shape= circle,
ball color = gray,text= black,inner sep= 0pt,outer sep= 0pt,minimum size= 6pt}} 
   \Vertex[LabelOut,Lpos=180,Ldist = .01cm,x=0 ,y=0,L=$\color{blue}\vec{v}$]{1}
     \Vertex[LabelOut,Lpos=0,Ldist =.01,x=3-.5 ,y=0,L=$\color{blue}\vec{v}+\vec{e}^{\,d}_1$]{2}
       \Vertex[LabelOut,Lpos=0,Ldist = .01cm,x=4-.5 ,y=1,L=$\color{blue}\vec{v}+\vec{e}^{\,d}_{1k}$]{3}
         \Vertex[LabelOut,Lpos=180,Ldist = .01cm,x=1 ,y=1,L=$\color{blue}\vec{v}+\vec{e}^{\,d}_k$]{4}
  
   \Vertex[LabelOut,Lpos=180,Ldist = .01cm,x=0 ,y=2,L=$\color{blue}\vec{v}+\vec{e}^{\,d}_j$]{5}
     \Vertex[LabelOut,Lpos=0,Ldist =.2cm,x=3-.5 ,y=2,L=$\color{blue}\vec{v}+\vec{e}^{\,d}_{1j}$]{6}
       \Vertex[LabelOut,Lpos=0,Ldist = .01cm,x=4-.5 ,y=3,L=$\color{blue}\vec{v}+\vec{e}^{\,d}_{1jk}$]{7}
         \Vertex[LabelOut,Lpos=180,Ldist = .01cm,x=1 ,y=3,L=$\color{blue}\vec{v}+\vec{e}^{\,d}_{jk}$]{8}
         
   \Edges[style={thick},color=black](1,4) 
     \Edges[style={thick},color=black](2,3) 
       \Edges[style={thick},color=black](5,8) 
         \Edges[style={thick},color=black](6,7) 
                      \Edges[style={-{stealth[length=4mm]},thick},color=red](1,2,6,5,1) 
                          \Edges[style={-{stealth[length=4mm]},thick},color=red](4,3,7,8,4)                 
    \end{tikzpicture}
    \end{center}
If $k \not\in\{1,j\}$, 
similar reasoning produces a 4-cycle with labels obtained by adding $\vec{e}^{\,d}_k$ to each label in the cycle on the left. Inductively, we can add an arbitrary sum 
 $\vec{v}=\sum_{k \not\in\{1,j\}} \vec{e}^{\,d}_k$ to each label,  using
the diagram on the right as a guide.  This yields 4-cycles based on  $\vec{e}^{\,d}_1$ and
 $\vec{e}^{\,d}_j$ covering all of $Q^0_d$ and
exhibiting the rotation $ P_1 P_{1j}$ explicitly. 
\end{proof}

\begin{lemma}\label{oppfaces}{\rm (Cancellation Across Opposite Faces).}
Let $\sigma$ be a nondegenerate singular $d+1$ cube with exactly two faces  supported by a $d$-cube $H$ in $G$, and suppose that these faces are opposite, say $f_k^-\sigma$ and $f_k^+\sigma$. Then $f_k^+\sigma$ is obtained from 
$f_k^-\sigma$ by a rotation of its domain by an element of  $\HH_d$, specifically, a quarter-turn  $P_i P_{ij}$ through one of its $2$-dimensional faces. 
Hence the contributions of $f_k^-\sigma$ and $f_k^+\sigma$ cancel out in the
computation of $\chi(\partial\sigma)$.

\end{lemma}

\begin{proof} 
Faces  $f_k^-\sigma$ and $f_k^+\sigma$ define a bijective graph map $T: Q^0_d\to Q^0_d$ satisfying the conditions
of Proposition \ref{pert}.  Clearly $T$ is not the identity, since $\sigma$ is not degenerate. Hence by Proposition \ref{pert}, either
$T = P_i P_{ij}$ or $T = P_i$, for some $i,j$. An easy argument shows that in the latter case there are four faces supported by $H$, 
not two as was assumed. Hence $T = P_i P_{ij}$ and since $\chi(f_k^-\sigma) = \chi(f_k^+\sigma) = +1$, the contributions of these
faces cancel out in the computation of $\chi(\partial\sigma)$.
\end{proof}

Combining Lemmas \ref{adjfaces} and \ref{oppfaces}, we can now prove Lemma \ref{chilemma} (Main Lemma).

\medskip
\noindent
{\bf Proof of Lemma \ref{chilemma}.} By Lemma \ref{lemma0124}, if $\sigma$ has any faces supported by $H$, then there are just three possible cases: 
\begin{enumerate}
\item [(1)]
a single adjacent pair, $f_i^{\epsilon_i} \sigma, f_j^{\epsilon_j} \sigma$, $i\ne j$,
\item[(2)]
a $4$-cycle of faces, each adjacent to the next,
\item[(3)]
a single opposite pair,
$f_i^{- }\sigma, f_i^{+} \sigma$.
\end{enumerate}
Applying $\chi$ to $\partial \sigma$ leaves either $2$ or $4$ nonzero terms, according to these cases.  Lemma \ref{adjfaces} shows that these terms cancel in cases (1) and (2). Lemma \ref{oppfaces} shows that they cancel in case (3). This completes the proof of Lemma \ref{chilemma}. \qed

The first main theorem is now an easy consequence.

\smallskip
\noindent
{\bf Proof of Theorem \ref{maintheorem} (First Main Theorem).} 
The first statement is an immediate consequence of Lemma \ref{chilemma}. 
To prove the second statement, suppose that
$\alpha: (I_\infty^d, \partial I_\infty^d) \to (G,v_0)$, and 
$h:I_\infty^d \times I_m \to G$ is a homotopy of length $m$ from $\alpha$ to the zero map. If we count 
(with multiplicity) the number of occurrences of $d$-cubes supported by $H$ as faces of $(d+1)$-cubes
in $\alpha(I_\infty^d \times I_m)$, this number is odd:  it is odd on the top layer $ I_\infty^d \times \{0\}$, even on the bottom layer $ I_\infty^d \times \{m\}$, and 
occurrences in the middle layers come in pairs. However, by hypothesis and Lemma \ref{lemma0124}, the number on each $(d+1)$-cube is even. 
This is a contradiction. \qed

We proceed next to the proof of Lemma \ref{gridlemma}. The argument uses essential lemmas from arguments in \cite{JAM}, \cite{JSA},
used originally to prove a different but related result. 

\smallskip
\noindent
{\bf Proof of Theorem \ref{gridlemma} (Second Main Theorem).} 
 Let $H$ be a rigid $d$-cube in $G$ with a grid neighborhood  $N_G(H)$,
identified with a neighborhood of the unit 
cube $Q^0_d = \{0,1\}^{\,d}$ in $\ZZ^d$. To reduce notational confusion, we denote elements of $N_G(H)$
by $h_{S^*} = \sum _{ k \in S^*} \pm  \vec{e}^{\,d} _{|k|}$, $S^* \subseteq \{\pm 1,\dots,\pm d\}$.
For example, $h_{-1} = - \vec{e}^{\,d}_1 = - h_1$ and $h_\emptyset = \vec{0}^d$. 
 As before, $Q_{d+1}$ is identified with the unit cube $Q^0_{d+1}=\{0,1\}^{d+1}$ in $\ZZ^{d+1}$,
 with vertices $\vec{e}^{\,d+1}_S$, $S \subseteq \{1,\dots,d+1\}$.  
 
 We will repeatedly use the fact that, in a grid graph of any dimension, 
 two distinct vertices $\vec{u}$ and $\vec{v}$ can have at most two common neighbors. 

Let $\sigma: Q^0_{d+1} \to G$ be a singular $(d+1)$-cube in which $f^-_{d+1}\sigma$ is supported by $H$. We will show that $\sigma$ has at 
least one more face supported by $H$. Suppose, to the contrary, that there is no such face. In particular, $f^+_{d+1}\sigma$ is not supported by $H$, so some element in $ H$ does not appear in $\Im(f^+_{d+1}\sigma)$.   
Assume, without loss of generality, that 
$h_\emptyset\in H$ is such an element, with $\sigma(\vec{0}^{d+1}) = h_\emptyset= \vec{0}^d$ and also
$\sigma(\vec{e}^{\,_{d+1}}_{i}) = h_i$ for $i = 1,\dots,d$. Note that, under this assumption, $h_\emptyset$ appears only once in $\Im(\sigma)$.

We first show that, if  $x= \sigma(\vec{e}^{\,_{d+1}}_{d+1})$, then  $x \in \{h_1,\dots,h_d\}$. Otherwise
$x$ would be another vertex in $N_G(H)$ adjacent to $h_\emptyset$, i.e.,  $x  \in \{-h_2,\dots,-h_d\}$.  But if 
$x = -  h_i$, then $\sigma(\vec{e_i} + \vec{e_{d+1}}) =h_\emptyset$ is forced, since $h_\emptyset$ is the only common neighbor of $h_i$ and $-h_i$ in $N_G(H)$:
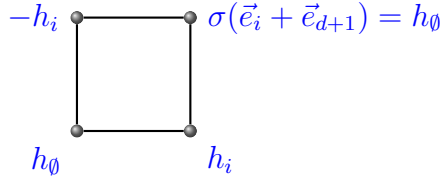
  \begin{center}
     \begin{tikzpicture}[scale=1.5]
\tikzset{VertexStyle/.style = {shape= circle,
ball color = gray,text= black,inner sep= 0pt,outer sep= 0pt,minimum size= 5pt}} 
   \Vertex[LabelOut,Lpos=225,x=0 ,y=0,L=${\color{blue}h_\emptyset}$]{1}
     \Vertex[LabelOut,Lpos=325,x=1 ,y=0,L=${\color{blue}h_i}$]{2}
       \Vertex[LabelOut,Lpos=180,x=0 ,y=1,L=${\color{blue}-h_i}$]{3}
         \Vertex[LabelOut,Lpos=0,x=1 ,y=1,L=${\color{blue} \sigma(\vec{e}_i+\vec{e}_{d+1})=h_\emptyset }$]{4}                   
    \Edge[style={thick},color=black](1)(2)
      \Edge[style={thick},color=black](3)(4)
        \Edge[style={thick},color=black](1)(3)
          \Edge[style={thick},color=black](2)(4)      
          \end{tikzpicture} 
    \end{center}
This contradicts the assumption that $h_\emptyset$ occurs only once in $\sigma$. 
Thus $\sigma(\vec{e}^{\,_{d+1}}_{d+1})= h_k$, for some $k \in \{1,\dots,d\}$. 
Using this value of $k$, we will prove that
 that $f_k^-(\sigma)$ is a second face supported by $H$.
 
The argument will show (by induction on $|S|$) that for all subsets
$S \subseteq \{1,2,\dots, \hat{k}, \dots,d\}$, we have 
\begin{equation}\label{remainingterms}
\sigma( \vec{e}^{\,d+1}_S+ \vec{e}^{\,d+1}_{k})=\sigma( \vec{e}^{\,d+1}_S+ \vec{e}^{\,d+1}_{d+1})= h_{S\cup k}. 
\end{equation}
Note that
the values of $\sigma( \vec{e}^{\,d+1}_S)$ range over the common face of $f^-_{d+1}\sigma$ and $f^-_{k}\sigma$; 
in (\ref{remainingterms}) we are showing that the remaining values on those  two faces agree.  
Figure \ref{figure1} above represents an example where a singular 4-cube $\sigma$ supports a $3$-cube $H$ on  faces, $f_4^-(\sigma)$ and $f_2^-(\sigma)$, with their 
common 2-face shown in black.  The task of the proof in this case is to show that labels on the faces shown in red 
(belonging to $f^-_{4}\sigma$) and blue (belonging to $f^-_{2}\sigma$) agree.
\begin{figure}
\begin{center}
     \begin{tikzpicture}[scale=1.5]
\tikzset{VertexStyle/.style = {shape= circle,
ball color = gray,text= black,inner sep= 0pt,outer sep= 0pt,minimum size= 5pt}} 
   \Vertex[LabelOut,Lpos=225,x=0 ,y=0,L=${\color{black}h_\emptyset}$]{1}
     \Vertex[LabelOut,Lpos=325,x=1 ,y=0,L=${\color{black}h_1}$]{2}
       \Vertex[LabelOut,Lpos=180,x=.5 ,y=.5,L=${\color{red}h_2}$]{3}
         \Vertex[LabelOut,Lpos=0,x=1.5 ,y=.5,L=${\color{red}h_{12}}$]{4}
            \Vertex[LabelOut,Lpos=180,x=0 ,y=1,L=${\color{black}h_3}$]{5}
               \Vertex[LabelOut,Lpos=0,x=1 ,y=1,L=${\color{black}h_{13}}$]{6}
                  \Vertex[LabelOut,Lpos=90,x=.5 ,y=1.5,L=${\color{red}h_{23}}$]{7}
                     \Vertex[LabelOut,Lpos=90,x=1.5 ,y=1.5,L=${\color{red}h_{123}}$]{8}
    \Vertex[LabelOut,Lpos=185,x=2.5 ,y=1.5,L=${\color{blue}h_2}$]{9}
     \Vertex[LabelOut,Lpos=325,x=3.5 ,y=1.5,L=${\color{blue}h_{12}}$]{10}
       \Vertex[LabelOut,Lpos=180,x=3 ,y=2,L=${\color{blue}}$]{11}
         \Vertex[LabelOut,Lpos=0,x=4 ,y=2,L=${\color{blue}}$]{12}
            \Vertex[LabelOut,Lpos=180,x=2.5 ,y=2.5,L=${\color{blue}h_{23}}$]{13}
               \Vertex[LabelOut,Lpos=0,x=3.5 ,y=2.5,L=${\color{blue}h_{123}}$]{14}
                  \Vertex[LabelOut,Lpos=180,x=3 ,y=3,L=${\color{blue}}$]{15}
                     \Vertex[LabelOut,Lpos=0,x=4 ,y=3,L=${\color{blue}}$]{16}

    \Edge[style={line width = 2pt},color=black](1)(2)
      \Edge[style={line width = 2pt},color=red](3)(4)
        \Edge[style={thick},color=black](1)(3)
          \Edge[style={thick},color=black](2)(4)      
      \Edge[style={line width = 2pt},color=red](7)(8)
        \Edge[style={thick},color=black](5)(7)
          \Edge[style={thick},color=black](6)(8)
          \Edge[style={line width = 2pt},color=black](1)(5)
      \Edge[style={line width = 2pt},color=black](2)(6)
        \Edge[style={line width = 2pt},color=red](3)(7)
           \Edge[style={line width = 2pt},color=black](5)(6)
          \Edge[style={ line width = 2pt},color=red](4)(8)        
            \Edge[style={line width = 2pt},color=blue](9)(10)
      \Edge[style={thick},color=black](11)(12)
        \Edge[style={thick},color=black](9)(11)
          \Edge[style={thick},color=black](10)(12)
           \Edge[style={line width = 2pt},color=blue](13)(14)
      \Edge[style={thick},color=black](15)(16)
        \Edge[style={thick},color=black](13)(15)
          \Edge[style={thick},color=black](14)(16)
          \Edge[style={line width = 2pt},color=blue](9)(13)
      \Edge[style={line width = 2pt},color=blue](10)(14)
        \Edge[style={thick},color=black](11)(15)
          \Edge[style={ thick},color=black](12)(16)
    \end{tikzpicture}
    \end{center}
\caption{Singular 4-cube supporting a 3-cube $H$ on faces $f_4^-$ and $f_2^-$, with $k=2$.}
\label{figure1}
\end{figure}
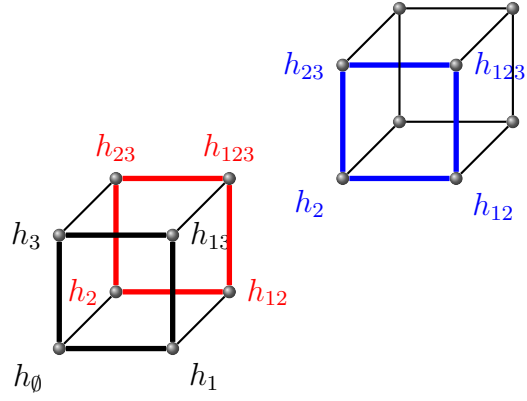
For $|S|=0$, the statement follows from the definition of $ k$.  For $|S|=1$, say $S =\{j\}$, 
we know that $\sigma(\vec{e}^{\,d+1}_j+ \vec{e}^{\,d+1}_{d+1})$ is a common neighbor of  
$\sigma(\vec{e}^{\,d+1}_j)=h_j$ and  $\sigma(\vec{e}^{\,d+1}_{d+1})=h_k$ in $H$, which forces   $\sigma(\vec{e}^{\,d+1}_j+ \vec{e}^{\,d+1}_{d+1})$ to
be either $\vec{0}^d$ or $h_{jk}$. But $\vec{0}^d$ appears only once in $\Im(\sigma)$,  hence $\sigma(\vec{e}^{\,d+1}_j+ \vec{e}^{\,d+1}_{d+1}) = h_{jk}$ as claimed.

  Finally, suppose $|S| \ge 2$, say $S = \{s_1,\dots,s_j\} \subseteq \{1,2,\dots,\hat{k},\dots,d\}$ with $j \ge 2$. For each 2-subset 
  $\{p,q\} \subseteq S\cup k$, there is a 2-face of $f_k^-(\sigma)$ of the form
  
  
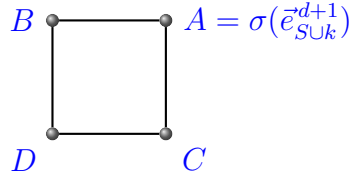
\begin{center}
     \begin{tikzpicture}[scale=1.5]
\tikzset{VertexStyle/.style = {shape= circle,
ball color = gray,text= black,inner sep= 0pt,outer sep= 0pt,minimum size= 5pt}} 
   \Vertex[LabelOut,Lpos=225,x=0 ,y=0,L=${\color{blue}D}$]{1}
     \Vertex[LabelOut,Lpos=325,x=1 ,y=0,L=${\color{blue}C}$]{2}
       \Vertex[LabelOut,Lpos=180,x=0 ,y=1,L=${\color{blue}B}$]{3}
         \Vertex[LabelOut,Lpos=0,x=1 ,y=1,L=${\color{blue}A= \sigma(\vec{e}^{\,d+1}_{S\cup k })}$]{4}                   
    \Edge[style={thick},color=black](1)(2)
      \Edge[style={thick},color=black](3)(4)
        \Edge[style={thick},color=black](1)(3)
          \Edge[style={thick},color=black](2)(4)      
          \end{tikzpicture} 
    \end{center}
 where 
 $A= \sigma(\vec{e}^{\,d+1}_{S\cup k })$,  
 and the remaining values $B,C$, and $D$ have been determined to satisfy (\ref{remainingterms}). 
By the inductive hypothesis, 
 $B, C$, and $D$ agree with the corresponding values of $\sigma$ with $k$ replaced by $d+1$, which lie on 
 a 2-face of $H$ with $D$ a common neighbor of $B$ and $C$.  Since $A$ is a common neighbor of $B$
 and $C$, the geometry of $N_G(H)$ forces $A$ to be either $\sigma(\vec{e}^{\,d+1}_{S\cup (d+1) })$ or 
 $\sigma(\vec{e}^{\,d+1}_{S\cup (d+1)-\{p,q\} })$. However, since the latter values are different for each choice of
 $p$ and $q$,  the only consistent choice is $A = \sigma(\vec{e}^{\,d+1}_{S\cup (d+1) }) = h_S$ for all $p,q$, and the induction is complete.
 \qed

\section{ Monophobicity in Dimension $2$}

\begin{theorem} \label{dim2theorem}
Let $H$
denote a rigid $2$-cube in $G$. Then $N_G(H)$  is monophobic if and only if it contains no
subgraphs isomorphic (by an $H$-preserving map) to one of the following graphs.

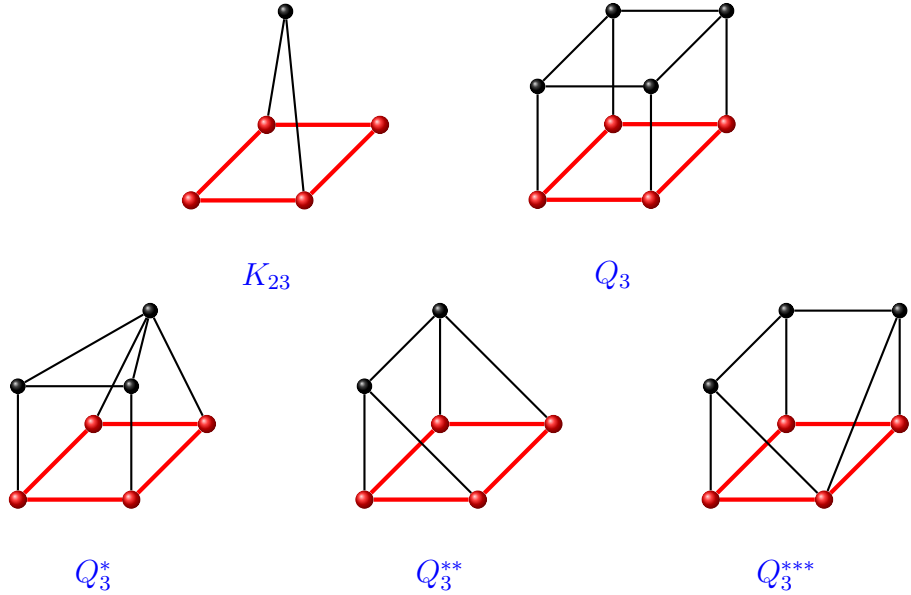
\begin{center}
 \begin{tikzpicture}[scale=.5]
\tikzset{VertexStyle/.style = {shape= circle,
ball color = red,inner sep= 0pt,outer sep= 0pt,minimum size= 7 pt}}   
\Vertex[x=2 ,y=2,L=$$]{1}
   \Vertex[x=5 ,y=2,L=$$]{2}
   \Vertex[x=4 ,y=4,L=$$]{3}
   \Vertex[x=7 ,y=4,L=$$]{4}
   \tikzset{VertexStyle/.style = {shape= circle,
ball color =  black,inner sep= 0pt,outer sep= 0pt,minimum size=6 pt}}     
   \Vertex[x=4.5 ,y=7,L=$$]{5}
      \Edge[style={ultra thick},color=red](1)(2)  \Edge[style={ultra thick},color=red](2)(4)  \Edge[style={ultra thick},color=red](3)(4)
    \Edge[style={ultra thick},color=red](1)(3)       
   \Edge(2)(5)
      \Edge(3)(5)
     \node  [color=blue] at (4,0) {$K_{23}$};
  \end{tikzpicture}
  \qquad\qquad
  \begin{tikzpicture}[scale=.5]
\tikzset{VertexStyle/.style = {shape= circle,
ball color = red,inner sep= 0pt,outer sep= 0pt,minimum size= 7 pt}}   
\Vertex[x=2 ,y=2,L=$$]{1}
   \Vertex[x=5 ,y=2,L=$$]{2}
   \Vertex[x=4 ,y=4,L=$$]{3}
   \Vertex[x=7 ,y=4,L=$$]{4}
   \tikzset{VertexStyle/.style = {shape= circle,
ball color =  black,inner sep= 0pt,outer sep= 0pt,minimum size=6 pt}}     
  \Vertex[x=2 ,y=5,L=$$]{5}
   \Vertex[x=5 ,y=5,L=$$]{6}
   \Vertex[x=4 ,y=7,L=$$]{7}
   \Vertex[x=7 ,y=7,L=$$]{8}
      \Edge[style={ultra thick},color=red](1)(2)  \Edge[style={ultra thick},color=red](2)(4)  \Edge[style={ultra thick},color=red](3)(4)
    \Edge[style={ultra thick},color=red](1)(3)       
   \Edge(1)(5)
      \Edge(2)(6)
       \Edge(3)(7)
        \Edge(4)(8)
   \Edge(5)(6)
    \Edge(7)(8)
     \Edge(5)(7)
      \Edge(6)(8)
     \node  [color=blue] at (4,0) {$Q_3$};
  \end{tikzpicture}
  
  \begin{tikzpicture}[scale=.5]
\tikzset{VertexStyle/.style = {shape= circle,
ball color = red,inner sep= 0pt,outer sep= 0pt,minimum size= 7 pt}}   
\Vertex[x=2 ,y=2,L=$$]{1}
   \Vertex[x=5 ,y=2,L=$$]{2}
   \Vertex[x=4 ,y=4,L=$$]{3}
   \Vertex[x=7 ,y=4,L=$$]{4}
   \tikzset{VertexStyle/.style = {shape= circle,
ball color =  black,inner sep= 0pt,outer sep= 0pt,minimum size=6 pt}}     
  \Vertex[x=2 ,y=5,L=$$]{5}
   \Vertex[x=5 ,y=5,L=$$]{6}
   \Vertex[x=11/2 ,y=7,L=$$]{7}
      \Edge[style={ultra thick},color=red](1)(2)  \Edge[style={ultra thick},color=red](2)(4)  \Edge[style={ultra thick},color=red](3)(4)
    \Edge[style={ultra thick},color=red](1)(3)       
   \Edge(1)(5)
      \Edge(2)(6)
       \Edge(3)(7)
   \Edge(5)(6)
     \Edge(5)(7)
      \Edge(6)(7)
      \Edge(4)(7)
     \node  [color=blue] at (4,0) {$Q_3^*$};
  \end{tikzpicture}
  \qquad\qquad
  \begin{tikzpicture}[scale=.5]
\tikzset{VertexStyle/.style = {shape= circle,
ball color = red,inner sep= 0pt,outer sep= 0pt,minimum size= 7 pt}}   
\Vertex[x=2 ,y=2,L=$$]{1}
   \Vertex[x=5 ,y=2,L=$$]{2}
   \Vertex[x=4 ,y=4,L=$$]{3}
   \Vertex[x=7 ,y=4,L=$$]{4}
   \tikzset{VertexStyle/.style = {shape= circle,
ball color =  black,inner sep= 0pt,outer sep= 0pt,minimum size=6 pt}}     
  \Vertex[x=2 ,y=5,L=$$]{5}
   \Vertex[x=4 ,y=7,L=$$]{7}
      \Edge[style={ultra thick},color=red](1)(2)  \Edge[style={ultra thick},color=red](2)(4)  \Edge[style={ultra thick},color=red](3)(4)
    \Edge[style={ultra thick},color=red](1)(3)       
   \Edge(1)(5)
         \Edge(3)(7)
     \Edge(5)(7)
       \Edge(2)(5)
         \Edge(4)(7)
     \node  [color=blue] at (4,0) {$Q_3^{**}$};
  \end{tikzpicture}
  \qquad\qquad
  \begin{tikzpicture}[scale=.5]
\tikzset{VertexStyle/.style = {shape= circle,
ball color = red,inner sep= 0pt,outer sep= 0pt,minimum size= 7 pt}}   
\Vertex[x=2 ,y=2,L=$$]{1}
   \Vertex[x=5 ,y=2,L=$$]{2}
   \Vertex[x=4 ,y=4,L=$$]{3}
   \Vertex[x=7 ,y=4,L=$$]{4}
   \tikzset{VertexStyle/.style = {shape= circle,
ball color =  black,inner sep= 0pt,outer sep= 0pt,minimum size=6 pt}}     
  \Vertex[x=2 ,y=5,L=$$]{5}
   \Vertex[x=4 ,y=7,L=$$]{7}
   \Vertex[x=7 ,y=7,L=$$]{8}
      \Edge[style={ultra thick},color=red](1)(2)  \Edge[style={ultra thick},color=red](2)(4)  \Edge[style={ultra thick},color=red](3)(4)
    \Edge[style={ultra thick},color=red](1)(3)       
   \Edge(1)(5)
         \Edge(3)(7)
        \Edge(4)(8)
    \Edge(7)(8)
     \Edge(5)(7)
        \Edge(2)(5)
              \Edge(2)(8)
     \node  [color=blue] at (4,0) {$Q_3^{***}$};
  \end{tikzpicture}
  \end{center}

\end{theorem}

\begin{proof}
It can easily be verified that each of the above graphs permits the construction of a monotolerant witness, that is, a singular $3$-cube with $H$ is supported on only one face. If $N_G(H)$ contains such a graph, it is not monophobic.
For example, suppose that $V(H)=\{1,2,3,4\}$ and $K_{23}$ appears in $N_G(H)$ with labels as shown:
\medskip
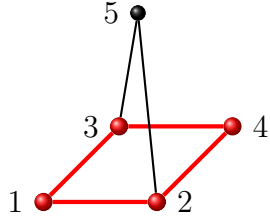
\begin{center}
 \begin{tikzpicture}[scale=.5]
\tikzset{VertexStyle/.style = {shape= circle,
ball color = red,inner sep= 0pt,outer sep= 0pt,minimum size= 7 pt}}   
\Vertex[LabelOut,Lpos = 180,x=2 ,y=2,L=${\color{black}1}$]{1}
   \Vertex[LabelOut,Lpos = 0,x=5 ,y=2,L=${\color{black}2}$]{2}
   \Vertex[LabelOut,Lpos = 180,x=4 ,y=4,L=${\color{black}3}$]{3}
   \Vertex[LabelOut,Lpos = 0,x=7 ,y=4,L=${\color{black}4}$]{4}
   \tikzset{VertexStyle/.style = {shape= circle,
ball color =  black,inner sep= 0pt,outer sep= 0pt,minimum size=6 pt}}     
   \Vertex[LabelOut,Lpos = 180,x=4.5 ,y=7,L=${\color{black}5}$]{5}
      \Edge[style={ultra thick},color=red](1)(2)  \Edge[style={ultra thick},color=red](2)(4)  \Edge[style={ultra thick},color=red](3)(4)
    \Edge[style={ultra thick},color=red](1)(3)       
   \Edge(2)(5)
      \Edge(3)(5)
  \end{tikzpicture}
  \end{center}
Then 
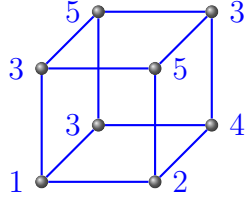
\begin{center}
\begin{tikzpicture}[scale=1.5]
\tikzset{VertexStyle/.style = {shape= circle,
ball color = gray,text= black,inner sep= 0pt,outer sep= 0pt,minimum size= 5pt}} 
   \Vertex[LabelOut,Lpos=180,x=0 ,y=0,L=${\color{blue}1}$]{1}
     \Vertex[LabelOut,Lpos=0,x=1 ,y=0,L=${\color{blue}2}$]{2}
       \Vertex[LabelOut,Lpos=180,x=.5 ,y=.5,L=${\color{blue}3}$]{3}
         \Vertex[LabelOut,Lpos=0,x=1.5 ,y=.5,L=${\color{blue}4}$]{4}
            \Vertex[LabelOut,Lpos=180,x=0 ,y=1,L=${\color{blue}3}$]{5}
               \Vertex[LabelOut,Lpos=0,x=1 ,y=1,L=${\color{blue}5}$]{6}
                  \Vertex[LabelOut,Lpos=180,x=.5 ,y=1.5,L=${\color{blue}5}$]{7}
                     \Vertex[LabelOut,Lpos=0,x=1.5 ,y=1.5,L=${\color{blue}3}$]{8}
    \Edge[style={thick},color=blue](1)(2)
      \Edge[style={thick},color=blue](3)(4)
        \Edge[style={thick},color=blue](1)(3)
          \Edge[style={thick},color=blue](2)(4)
           \Edge[style={thick},color=blue](5)(6)
      \Edge[style={thick},color=blue](7)(8)
        \Edge[style={thick},color=blue](5)(7)
          \Edge[style={thick},color=blue](6)(8)
          \Edge[style={thick},color=blue](1)(5)
      \Edge[style={thick},color=blue](2)(6)
        \Edge[style={thick},color=blue](3)(7)
          \Edge[style={ thick},color=blue](4)(8)
    \end{tikzpicture}
    \end{center}
is a singular 3-cube witnessing non-monophobicity (=monotolerance).

Conversely, if  $N_G(H)$ is not monophobic,  there exists a singular cube $\sigma =$
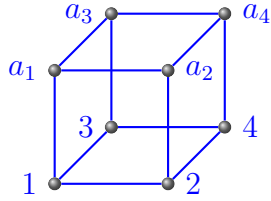
\begin{center}
\begin{tikzpicture}[scale=1.5]
\tikzset{VertexStyle/.style = {shape= circle,
ball color = gray,text= black,inner sep= 0pt,outer sep= 0pt,minimum size= 5pt}} 
   \Vertex[LabelOut,Lpos=180,x=0 ,y=0,L=${\color{blue}1}$]{1}
     \Vertex[LabelOut,Lpos=0,x=1 ,y=0,L=${\color{blue}2}$]{2}
       \Vertex[LabelOut,Lpos=180,x=.5 ,y=.5,L=${\color{blue}3}$]{3}
         \Vertex[LabelOut,Lpos=0,x=1.5 ,y=.5,L=${\color{blue}4}$]{4}
            \Vertex[LabelOut,Lpos=180,x=0 ,y=1,L=${\color{blue}a_1}$]{5}
               \Vertex[LabelOut,Lpos=0,x=1 ,y=1,L=${\color{blue}a_2}$]{6}
                  \Vertex[LabelOut,Lpos=180,x=.5 ,y=1.5,L=${\color{blue}a_3}$]{7}
                     \Vertex[LabelOut,Lpos=0,x=1.5 ,y=1.5,L=${\color{blue}a_4}$]{8}
    \Edge[style={thick},color=blue](1)(2)
      \Edge[style={thick},color=blue](3)(4)
        \Edge[style={thick},color=blue](1)(3)
          \Edge[style={thick},color=blue](2)(4)
           \Edge[style={thick},color=blue](5)(6)
      \Edge[style={thick},color=blue](7)(8)
        \Edge[style={thick},color=blue](5)(7)
          \Edge[style={thick},color=blue](6)(8)
          \Edge[style={thick},color=blue](1)(5)
      \Edge[style={thick},color=blue](2)(6)
        \Edge[style={thick},color=blue](3)(7)
          \Edge[style={ thick},color=blue](4)(8)
    \end{tikzpicture}
    \end{center}
 in which $H$ is supported only by one face (in this case, the bottom face). We will show that, in this case, $N_G(H)$ must contain 
 one of the graphs listed above. Specifically, we claim that if  $K_{23}$ does not occur in $N_G(H)$, then $\sigma$ 
 represents a {\em straight collapse onto $H$}, which means (by definition) that for $i=1,\dots,4$, either $a_i=i$ or 
 $a_i \not\in H$. If this condition is satisfied, it is easy to show that one of the five graphs in the statement of the theorem must appear: if $a_i\not\in H$ 
 for all $i$, one obtains either $Q_3$ or $Q_3^*$, and otherwise one obtains either $Q_3^{**}$ or $Q_3^{***}$.

To prove the claim,  we show that  if $a_i=H$ and $a_i\ne i$ for some $i$, there is a second face supported by $H$.
For example, if $a_1\in H$, then $a_1\ne 4$ by rigidity. If $a_1=2$, consider the possibilities for $a_3$. We have
$a_3\ne 2$ (by rigidity), and $a_3 \ne 4$ (this creates another face supported by $H$). Finally, $a_3\not\in H$ is not possible, since this
creates a subgraph of $N_G(H)$ isomorphic to $K_{23}$. Hence $a_3 = 1$ is forced.  Walking around the top face, we can apply exactly analogous arguments to deduce $a_4 = 3$ and $a_2 = 4$. But this has produced a second face supported by $H$, contrary to our assumption about $\sigma$.  This establishes the claim, and completes the proof of Theorem \ref{dim2theorem}

\end{proof}

\section{ Sufficient Condition for Monophobicity in Dimensions $d \ge 2$}

One can check that the five neighborhoods listed in Theorem \ref{dim2theorem} are {\em minimally monotolerant},
that is, they become monophobic if one removes any edge in $N_G(H) - H$.  A complete list of minimal monotolerants  
would solve the characterization problem
in higher dimensions as well.
In this section we will define analogs of $K_{23}$ in dimension $d$ for $d>2$, and these would be on such a list. 
Analogs of the other neighborhoods  in Theorem \ref{dim2theorem} ($Q_3, Q_3^*, Q_3^{**}$, and $Q_3^{***}$) 
clearly exist in all dimensions.  However, despite this promising start, computations suggest that
 a complete list of minimal monotolerants will be prohibitively large, even in dimension 3. 
Instead we will settle for a sufficient (but not necessary for $d>2$) condition. Although
not a complete solution to the problem, it can be used to facilitate hand computation in some cases.

\begin{definition}\label{forcingatv} {\rm (Forcing at $h_0$)} Suppose $H$ is a rigid $d$-cube in $G$, and  $h_0\in H$. Then $N_G(H)$ 
satisfies {\em  forcing  at $h_0$} if, for every $p \in N_G(h_0) - H$,  there exists a vertex $h'\in (N_G(h_0) \cap H)-h_0$ such that 
any (hence all) of the following equivalent conditions holds:
\begin{itemize}
\item
every path
$p \simeq x \simeq h'$ implies  $x=h_0$,
\item
$N_G(H)$ does not contain either of the configurations
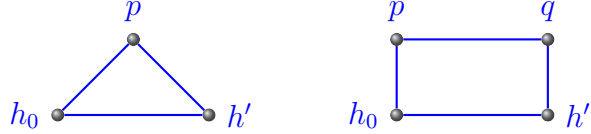
\begin{center}
\begin{tikzpicture}[scale=1]
\tikzset{VertexStyle/.style = {shape= circle,
ball color = gray,text= black,inner sep= 0pt,outer sep= 0pt,minimum size= 5pt}} 
   \Vertex[LabelOut,Lpos=180,x=0 ,y=0,L=${\color{blue}h_0}$]{1}
     \Vertex[LabelOut,Lpos=0,x=2 ,y=0,L=${\color{blue}h'}$]{2}
       \Vertex[LabelOut,Lpos=90,x=1 ,y=1,L=${\color{blue}p}$]{3}     
    \Edge[style={thick},color=blue](1)(2)
      \Edge[style={thick},color=blue](1)(3)
        \Edge[style={thick},color=blue](2)(3)
    \end{tikzpicture}
    \qquad
    \begin{tikzpicture}[scale=1]
\tikzset{VertexStyle/.style = {shape= circle,
ball color = gray,text= black,inner sep= 0pt,outer sep= 0pt,minimum size= 5pt}} 
   \Vertex[LabelOut,Lpos=180,x=0 ,y=0,L=${\color{blue}h_0}$]{1}
     \Vertex[LabelOut,Lpos=0,x=2 ,y=0,L=${\color{blue}h'}$]{2}
       \Vertex[LabelOut,Lpos=90,x=0 ,y=1,L=${\color{blue}p}$]{3}   
              \Vertex[LabelOut,Lpos=90,x=2 ,y=1,L=${\color{blue}q}$]{4}     
    \Edge[style={thick},color=blue](1)(2)
      \Edge[style={thick},color=blue](3)(4)
        \Edge[style={thick},color=blue](1)(3)
               \Edge[style={thick},color=blue](2)(4)
    \end{tikzpicture}
    \end{center}
\item
 $h_0$ is
 the unique common neighbor of $p$ and $h'$ (a vertex can be a neighbor of itself),
  \item
 if
 $\sigma=$
    \raisebox{-1.5em}{
\begin{tikzpicture}[scale=1]
\tikzset{VertexStyle/.style = {shape= circle,
ball color = gray,text= black,inner sep= 0pt,outer sep= 0pt,minimum size= 5pt}} 
   \Vertex[LabelOut,Lpos=180,x=0 ,y=0,L=${\color{red}h_0}$]{1}
     \Vertex[LabelOut,Lpos=0,x=1 ,y=0,L=${\color{red}h'}$]{2}
       \Vertex[LabelOut,Lpos=180,x=0 ,y=1,L=${\color{red}p}$]{3}
         \Vertex[LabelOut,Lpos=0,x=1 ,y=1,L=${\color{red}x}$]{4}
    \Edge[style={thick},color=blue](1)(2)
      \Edge[style={thick},color=blue](3)(4)
        \Edge[style={thick},color=blue](1)(3)
          \Edge[style={thick},color=blue](2)(4)
    \end{tikzpicture}
    }
is a singular 2-cube, then $x=h_0$.
 \end{itemize}
\end{definition}
\noindent
We view this as a condition that ``forces" an element $h_0$ on the bottom face of a singular cube to also appear on the top face.

\begin{theorem}{\rm (Sufficient Condition for Monophobicity)}\label{suffcondition}
Let  $H$ be a rigid $d$-cube in $G$, with $d \ge 2$.  If $N_G(H)$ contains no subgraphs isomorphic to $K_{23}$, then 
$H$ is monophobic in $G$ if for every $h_0\in H$, either
\begin{enumerate}
\item   forcing  is satisfied at $h_0$, or
\item 
 forcing  fails at $h_0$, but 
 for every $p\in N_G(h_0)-H$ there exists a pair of neighbors $h_1, h_2$ of $h_0$ such that for every $x_1,x_2$ 
with 
paths $p \simeq x_{1}\simeq h_1$ and  $p \simeq x_{2}\simeq h_2$, $x_{1},x_{2}\ne h_0$, the set
$\{x_{1},x_{ 2}, h_{12}\}$ has no common neighbor.
(Here $h_{12}$ denotes the fourth element of the 2-face of $H$ containing $h_0, h_1$, and $h_2$,
and $\simeq$ means that elements in the path may repeat.) 
\end{enumerate}
If $d\le2$, the conditions of this theorem are both necessary and sufficient; if $d>2$ they are sufficient.
\end{theorem}
\noindent
Note: condition (2) says that there
is no way to complete a singular $3$-cube with vertices as shown in Figure (\ref{condition2}). Alternatively, one can
express condition (2) as excluding certain of the 2-dimensional configurations appearing in Theorem \ref{maintheorem}. 
We leave details of this translation to the reader.
\begin{figure}
\begin{center}
\begin{tikzpicture}[scale=1.5]
\tikzset{VertexStyle/.style = {shape= circle,
ball color = gray,text= black,inner sep= 0pt,outer sep= 0pt,minimum size= 5pt}} 
   \Vertex[LabelOut,Lpos=180,x=0 ,y=0,L=${\color{blue}h_0}$]{1}
     \Vertex[LabelOut,Lpos=0,x=1 ,y=0,L=${\color{blue}h_1}$]{2}
       \Vertex[LabelOut,Lpos=180,x=.5 ,y=.5,L=${\color{blue}h_2}$]{3}
         \Vertex[LabelOut,Lpos=0,x=1.5 ,y=.5,L=${\color{blue}h_{12}}$]{4}
            \Vertex[LabelOut,Lpos=180,x=0 ,y=1,L=${\color{blue}p}$]{5}
               \Vertex[LabelOut,Lpos=0,x=1 ,y=1,L=${\color{blue}x_{1}}$]{6}i
                  \Vertex[LabelOut,Lpos=180,x=.5 ,y=1.5,L=${\color{blue}x_{2}}$]{7}
                     \Vertex[LabelOut,Lpos=0,x=1.5 ,y=1.5,L=${\color{blue}?}$]{8}
    \Edge[style={thick},color=blue](1)(2)
      \Edge[style={thick},color=blue](3)(4)
        \Edge[style={thick},color=blue](1)(3)
          \Edge[style={thick},color=blue](2)(4)
           \Edge[style={thick},color=blue](5)(6)
      \Edge[style={thick},color=blue](7)(8)
        \Edge[style={thick},color=blue](5)(7)
          \Edge[style={thick},color=blue](6)(8)
          \Edge[style={thick},color=blue](1)(5)
      \Edge[style={thick},color=blue](2)(6)
        \Edge[style={thick},color=blue](3)(7)
          \Edge[style={ thick},color=blue](4)(8)
    \end{tikzpicture}
    \end{center}
    \label{condition2}
    \caption{Condition (2).}
       \label{condition2}
    \end{figure}
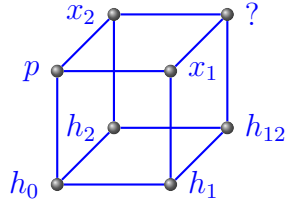

\begin{proof}{\rm (Sketch)}
The proof of Theorem \ref{gridlemma} assumes that $N_G(H)$ embeds in a grid neighborhood of $H$. However, a close reading of that proof  shows that only the hypotheses of  Theorem \ref{suffcondition} 
 are used. In fact, forcing always holds at every vertex of  $H$ in a grid neighborhood, and so only condition (1) is used in that case.

If $d=1$, monophobicity is equivalent to avoiding triangles and the result is elementary, so we may
assume that $d\ge 2$. Proceeding as in the proof of Theorem \ref{gridlemma}, we suppose that $\sigma $ is a singular $(d+1)$-cube with $f_{d+1}^-\sigma$ supported by $H$. 
Assuming the conditions of Theorem \ref{suffcondition}, the goal is to show that at least one more face is supported by $H$.  If 
not, there must be a vertex
$h_0\in H$ not appearing in $f_{d+1}^+\sigma$, and we may assume without loss of generality that $\sigma(\vec{0}) = h_0$. 

The first step will be to show that $\sigma(\vec{e}_{d+1}) = h_k$ for some $h_k$ adjacent to $h_0$. Suppose, otherwise,
that $\sigma(\vec{e}_{d+1}) = p$  for some $p \not\in H$.
If condition (1) holds at $h_0$, then
we have a contradiction, since then  $h_0$ is forced to appear in $f_{d+1}^+\sigma$.
If condition (1) fails at $h_0\in H$, then for every $h_k$, $1 \le k \le d$ there is a path 
$p \simeq x_{k} \simeq h_k$ with $x_{k} \ne h_0$. Condition (2) now yields a contradiction, since $\sigma$
must contain a 3-face with vertices labeled as above. Thus we conclude that $\sigma(\vec{e}_{d+1}) = h_k$ for some $h_k$.

Subsequent steps use only the fact that $N_G(H)$ contains no copies of $K_{23}$, and involve  2-faces of $f_{k}^-\sigma$ of the form
  
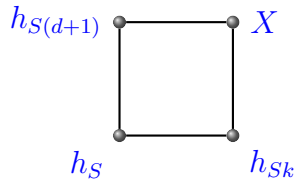
\begin{center}
     \begin{tikzpicture}[scale=1.5]
\tikzset{VertexStyle/.style = {shape= circle,
ball color = gray,text= black,inner sep= 0pt,outer sep= 0pt,minimum size= 5pt}} 
   \Vertex[LabelOut,Lpos=225,x=0 ,y=0,L=${\color{blue}h_S}$]{1}
     \Vertex[LabelOut,Lpos=325,x=1 ,y=0,L=${\color{blue}h_{Sk}}$]{2}
       \Vertex[LabelOut,Lpos=180,x=0 ,y=1,L=${\color{blue}h_{S(d+1)}}$]{3}
         \Vertex[LabelOut,Lpos=0,x=1 ,y=1,L=${\color{blue}X}$]{4}                   
    \Edge[style={thick},color=black](1)(2)
      \Edge[style={thick},color=black](3)(4)
        \Edge[style={thick},color=black](1)(3)
          \Edge[style={thick},color=black](2)(4)      
          \end{tikzpicture} 
    \end{center}
where three of the values (vertices in $H$) are known and the fourth remains to be determined. Since $N_G(H)$
contains no copies of $K_{23}$, it follows that $X$ must be either $h_{S}$ or $h_{Sk(d+1)}$. 
Once this is established, the proof proceeds exactly as in the proof of Theorem \ref{gridlemma} without change.

The proof of necessity for $d=2$ is elementary and left to the reader.
\end{proof}
As an application, note that  Example \ref{mpexamples}$(C)$ can be shown to be
monophobic by an argument 
applying Theorem \ref{suffcondition} in which only the condition (1)  is used.
The role of this condition can be seen by comparing Examples \ref{mpexamples}$(C)$ and \ref{nmpexamples}$(E)$.
\begin{center}
  \begin{tikzpicture}[scale=.7]
   \tikzset{VertexStyle/.style = {shape= circle,
ball color =  red,inner sep= 0pt,outer sep= 0pt,minimum size=5 pt}}     
\Vertex[LabelOut,Lpos=225,x=2 ,y=1.5,L=$\vec{0}$]{1}
   \Vertex[LabelOut,Lpos=315,x=5 ,y=1.5,L=$e_1$]{2}
   \Vertex[LabelOut,Lpos=135,x=2 ,y=4.5,L=$e_2$]{3}
   \Vertex[LabelOut,Lpos=45,x=5 ,y=4.5,L=$e_{12}$]{4}
   \tikzset{VertexStyle/.style = {shape= circle,
ball color =  black,inner sep= 0pt,outer sep= 0pt,minimum size=5 pt}}     
   \Vertex[LabelOut,Lpos=180,x=3.5 ,y=2.5,L=$p$]{5}
   \Vertex[LabelOut,Lpos=0,x=6.5 ,y=2.5,L=$q$]{6}
      \Edge[style={ultra thick},color=red](1)(2)  \Edge[style={ultra thick},color=red](2)(4)  \Edge[style={ultra thick},color=red](3)(4)
    \Edge[style={ultra thick},color=red](1)(3)       
    \Edge(1)(5)\Edge(3)(5)
      \Edge(2)(6)\Edge(4)(6)
     \node  [color=blue] at (4,0) {$(C)$};
  \end{tikzpicture}
  \qquad
    \begin{tikzpicture}[scale=.7]
   \tikzset{VertexStyle/.style = {shape= circle,
ball color =  red,inner sep= 0pt,outer sep= 0pt,minimum size=5 pt}}     
\Vertex[LabelOut,Lpos=225,x=2 ,y=1.5,L=$\vec{0}$]{1}
   \Vertex[LabelOut,Lpos=315,x=5 ,y=1.5,L=$e_1$]{2}
   \Vertex[LabelOut,Lpos=135,x=2 ,y=4.5,L=$e_2$]{3}
   \Vertex[LabelOut,Lpos=45,x=5 ,y=4.5,L=$e_{12}$]{4}
   \tikzset{VertexStyle/.style = {shape= circle,
ball color =  black,inner sep= 0pt,outer sep= 0pt,minimum size=5 pt}}     
   \Vertex[LabelOut,Lpos=180,x=3.5 ,y=2.5,L=$p$]{5}
   \Vertex[LabelOut,Lpos=0,x=6.5 ,y=2.5,L=$q$]{6}
      \Edge[style={ultra thick},color=red](1)(2)  \Edge[style={ultra thick},color=red](2)(4)  \Edge[style={ultra thick},color=red](3)(4)
    \Edge[style={ultra thick},color=red](1)(3)       
    \Edge(1)(5)\Edge(3)(5)
      \Edge(2)(6)\Edge(5)(6)
     \node  [color=blue] at (4,0) {$(C')$};
  \end{tikzpicture}
   \qquad
    \begin{tikzpicture}[scale=.7]
   \tikzset{VertexStyle/.style = {shape= circle,
ball color =  red,inner sep= 0pt,outer sep= 0pt,minimum size=5 pt}}     
\Vertex[LabelOut,Lpos=225,x=2 ,y=1.5,L=$\vec{0}$]{1}
   \Vertex[LabelOut,Lpos=315,x=5 ,y=1.5,L=$e_1$]{2}
   \Vertex[LabelOut,Lpos=135,x=2 ,y=4.5,L=$e_2$]{3}
   \Vertex[LabelOut,Lpos=45,x=5 ,y=4.5,L=$e_{12}$]{4}
   \tikzset{VertexStyle/.style = {shape= circle,
ball color =  black,inner sep= 0pt,outer sep= 0pt,minimum size=5 pt}}     
   \Vertex[LabelOut,Lpos=180,x=3.5 ,y=2.5,L=$p$]{5}
   \Vertex[LabelOut,Lpos=0,x=6.5 ,y=2.5,L=$q$]{6}
      \Edge[style={ultra thick},color=red](1)(2)  \Edge[style={ultra thick},color=red](2)(4)  \Edge[style={ultra thick},color=red](3)(4)
    \Edge[style={ultra thick},color=red](1)(3)       
    \Edge(1)(5)\Edge(3)(5)
      \Edge(2)(6)\Edge(4)(6)\Edge(5)(6)
     \node  [color=blue] at (4,0) {$(E)$};
  \end{tikzpicture}
  \end{center}
Example $(C)$ satisfies condition (1) at every vertex of $H$, and hence $N_G(H)$ is monophobic. 
Example $(C')$ satisfies condition (1) at every vertex except $\vec{0}$ but satisfies condition (2) there. Again we conclude  $N_G(H)$ is monophobic.
Example $(E)$ fails both (1) and (2) at every vertex. It is monotolerant, and admits a witness $\sigma$ in which $\vec{0}$ (for example) 
does not appear on the opposite face and is adjacent to a vertex not in $H$. In the diagram below, the label $q$ is
valid for $(E)$ but not for $(C)$ or $(C')$.
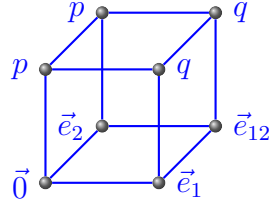
\begin{center}
\begin{tikzpicture}[scale=1.5]
\tikzset{VertexStyle/.style = {shape= circle,
ball color = gray,text= black,inner sep= 0pt,outer sep= 0pt,minimum size= 5pt}} 
   \Vertex[LabelOut,Lpos=180,x=0 ,y=0,L=${\color{blue}\vec{0}}$]{1}
     \Vertex[LabelOut,Lpos=0,x=1 ,y=0,L=${\color{blue}\vec{e}_1}$]{2}
       \Vertex[LabelOut,Lpos=180,x=.5 ,y=.5,L=${\color{blue}\vec{e}_2}$]{3}
         \Vertex[LabelOut,Lpos=0,x=1.5 ,y=.5,L=${\color{blue}\vec{e}_{12}}$]{4}
            \Vertex[LabelOut,Lpos=180,x=0 ,y=1,L=${\color{blue}p}$]{5}
               \Vertex[LabelOut,Lpos=0,x=1 ,y=1,L=${\color{blue}q}$]{6}
                  \Vertex[LabelOut,Lpos=180,x=.5 ,y=1.5,L=${\color{blue}p}$]{7}
                     \Vertex[LabelOut,Lpos=0,x=1.5 ,y=1.5,L=${\color{blue}q}$]{8}
    \Edge[style={thick},color=blue](1)(2)
      \Edge[style={thick},color=blue](3)(4)
        \Edge[style={thick},color=blue](1)(3)
          \Edge[style={thick},color=blue](2)(4)
           \Edge[style={thick},color=blue](5)(6)
      \Edge[style={thick},color=blue](7)(8) 
        \Edge[style={thick},color=blue](5)(7)
          \Edge[style={thick},color=blue](6)(8)
          \Edge[style={thick},color=blue](1)(5)
      \Edge[style={thick},color=blue](2)(6)
        \Edge[style={thick},color=blue](3)(7)
          \Edge[style={ thick},color=blue](4)(8)
    \end{tikzpicture}
    \end{center}
Using Theorem \ref{suffcondition}, one can quickly verify that all five of the graphs in Theorem \ref{dim2theorem} represent 
minimally monotolerant neighborhoods. 
One can also give simple arguments proving  monophobicity of two examples 
mentioned earlier and illustrated in Figure \ref{fig3} below: corner squares in the subdivided cube boundary (Example \ref{ex3.1}) and top squares in a box suspension (Example \ref{ex3.2}).  Neither one of these has a grid neighborhood. The first argument
 requires using condition (2) in Theorem \ref{suffcondition}, but the second uses only condition (1).
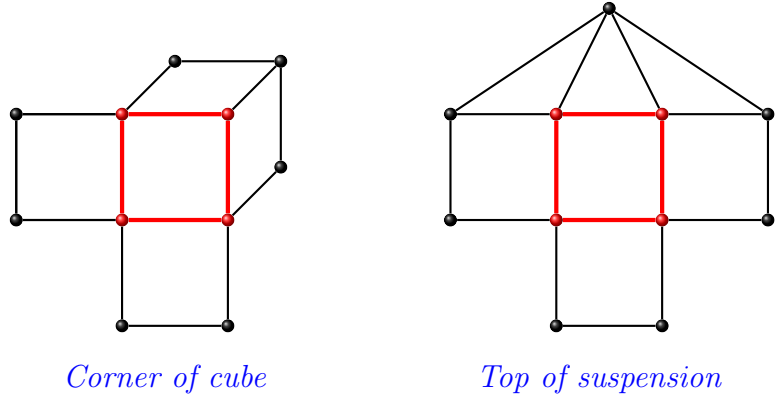
\begin{figure}
 \begin{center}
     \begin{tikzpicture}[scale=.7]
   \tikzset{VertexStyle/.style = {shape= circle,
ball color =  red,inner sep= 0pt,outer sep= 0pt,minimum size=5 pt}}     
\Vertex[LabelOut,Lpos=0,x=2 ,y=2,L=$$]{1}
   \Vertex[LabelOut,Lpos=0,x=4 ,y=2,L=$$]{2}
   \Vertex[LabelOut,Lpos=0,x=2 ,y=4,L=$$]{3}
   \Vertex[LabelOut,Lpos=0,x=4 ,y=4,L=$$]{4}
   \tikzset{VertexStyle/.style = {shape= circle,
ball color =  black,inner sep= 0pt,outer sep= 0pt,minimum size=5 pt}}     
   \Vertex[LabelOut,Lpos=0,x=0,y=2,L=$$]{5}
   \Vertex[LabelOut,Lpos=0,x=0 ,y=4,L=$$]{6}
    \Vertex[LabelOut,Lpos=0,x=3 ,y=5,L=$$]{7}
     \Vertex[LabelOut,Lpos=0,x=5 ,y=5,L=$$]{8}
      \Vertex[LabelOut,Lpos=0,x=5 ,y=3,L=$$]{9}
       \Vertex[LabelOut,Lpos=0,x=4 ,y=0,L=$$]{10}
        \Vertex[LabelOut,Lpos=0,x=2 ,y=0,L=$$]{11}   
    \Edge[style={ultra thick},color=red](1)(2)  \Edge[style={ultra thick},color=red](2)(4)  \Edge[style={ultra thick},color=red](3)(4)
    \Edge[style={ultra thick},color=red](1)(3)       
    \Edge(1)(5)
      \Edge(3)(6)\Edge(5)(6)
      \Edge(1)(5) \Edge(5)(6) \Edge(3)(6)
       \Edge(3)(7) \Edge(4)(8) \Edge(2)(9) \Edge(7)(8) \Edge(8)(9) 
        \Edge(1)(11) \Edge(2)(10) \Edge(10)(11)
     \node  at (2.8,-1)  {\color{blue}\it Corner of cube };
  \end{tikzpicture}
\qquad\qquad
     \begin{tikzpicture}[scale=.7]
   \tikzset{VertexStyle/.style = {shape= circle,
ball color =  red,inner sep= 0pt,outer sep= 0pt,minimum size=5 pt}}     
\Vertex[LabelOut,Lpos=0,x=2 ,y=2,L=$$]{1}
   \Vertex[LabelOut,Lpos=0,x=4 ,y=2,L=$$]{2}
   \Vertex[LabelOut,Lpos=0,x=2 ,y=4,L=$$]{3}
   \Vertex[LabelOut,Lpos=0,x=4 ,y=4,L=$$]{4}
   \tikzset{VertexStyle/.style = {shape= circle,
ball color =  black,inner sep= 0pt,outer sep= 0pt,minimum size=5 pt}}     
   \Vertex[LabelOut,Lpos=0,x=6,y=4,L=$$]{5}
   \Vertex[LabelOut,Lpos=0,x=6 ,y=2,L=$$]{6}
    \Vertex[LabelOut,Lpos=0,x=4 ,y=0,L=$$]{7}
     \Vertex[LabelOut,Lpos=0,x=2 ,y=0,L=$$]{8}
      \Vertex[LabelOut,Lpos=0,x=0 ,y=2,L=$$]{9}
       \Vertex[LabelOut,Lpos=0,x=0 ,y=4,L=$$]{10}
        \Vertex[LabelOut,Lpos=0,x=3 ,y=6,L=$$]{11}   
    \Edge[style={ultra thick},color=red](1)(2)  \Edge[style={ultra thick},color=red](2)(4)  \Edge[style={ultra thick},color=red](3)(4)
    \Edge[style={ultra thick},color=red](1)(3)       
        \Edge(5)(6)
      \Edge(1)(9) \Edge(1)(8) \Edge(2)(6) \Edge(2)(7)
       \Edge(4)(5) \Edge(4)(11) \Edge(3)(10) \Edge(3)(11)  \Edge(5)(11) 
        \Edge(10)(11)  \Edge(10)(9)  \Edge(7)(8)
     \node  at (2.8,-1)  {\color{blue}\it Top of suspension };
  \end{tikzpicture}
\end{center}
    \caption{Monophobic Neighborhoods}
       \label{fig3}
    \end{figure}

In fact Theorem \ref{suffcondition} can be used to prove much more. Consider graphs $G_K$ arising from {\em simple quadrangulations $K$ of a $2$-sphere}, i.e., graphs arising as the 
$1$-skeleton of a cellular decomposition of $S^2$  such that every $2$-cell is bounded by a nondegenerate $4$-cycle. (See \cite{BGGMTW} for precise definitions, and \cite{Stu} for illustrations of very large examples.) Then we have the following generalization of Examples \ref{ex3.3} and \ref{ex3.1}.
\begin{corollary}\label{s2cubes}
Suppose that $G$ arises from a simple quadrangulation of $S^2$.  If $G$ contains no subgraphs isomorphic to $K_{23}$ or $Q_3$, then every 
quadrangle in $G$ has a monophobic neighborhood. 
\end{corollary}
\begin{proof}
Absence of $K_{23}$ permits Theorem \ref{suffcondition} to be applied. In Theorem \ref{suffcondition}, 
condition (2) never applies, since $G$ can have no subgraphs isomorphic to $Q_3$ . Hence condition (1) is satisfied at every $h_0$ and the result follows.
\end{proof}
 \medskip
 
Motivated by the special role played by $K_{23}$ in characterizing monophobicity in dimension $2$, we 
briefly consider analogs of $K_{23}$ in higher dimensions.  

\begin{definition}\label{kdplus1}
\begin{enumerate}
\item
 For $d\ge 1$,
 $K^{\{d\}}$ is a $d$-cube  $Q_{d}$ plus a single vertex $v$ connected to vertices $\vec{0}$ (in $H$) and its antipodal vertex 
 $\vec{1}$ .
\item
For $d\ge 2$,
 $K_{23}^{\{d\}}$ is a $d$-cube  $Q_{d}$ plus a single vertex $v$ connected to vertices $\vec{e}_1, \vec{e}_2,\dots,\vec{e}_{d} $ in $Q_{d}$. 
 \end{enumerate}
 \end{definition}
 
 \begin{example}\label{kdplus1examples} $K^{\{d\}}, d = 1,2,3$.
\medskip
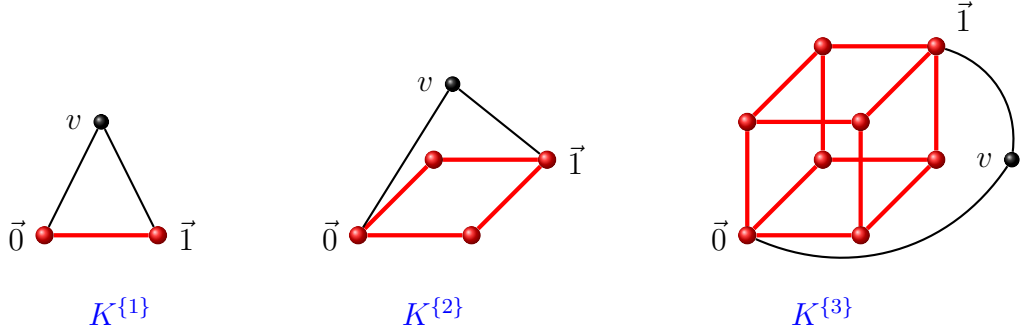
\begin{center}
 \begin{tikzpicture}[scale=.5]
\tikzset{VertexStyle/.style = {shape= circle,
ball color = red,inner sep= 0pt,outer sep= 0pt,minimum size= 7 pt}}   
\Vertex[LabelOut,Lpos=180,x=2 ,y=2,L=$\vec{0}$]{1}
   \Vertex[LabelOut,Lpos=0,x=5 ,y=2,L=$\vec{1}$]{2}
    \tikzset{VertexStyle/.style = {shape= circle,
ball color =  black,inner sep= 0pt,outer sep= 0pt,minimum size=6 pt}}     
   \Vertex[LabelOut,Lpos=180,x=3.5 ,y=5,L=$v$]{3}
      \Edge[style={ultra thick},color=red](1)(2)  \Edge[style={ thick},color=black](1)(3)  \Edge[style={ thick},color=black](2)(3)
       \node  [color=blue] at (4,0) {$K^{\{1\}}$};
  \end{tikzpicture}
  \qquad\quad
   \begin{tikzpicture}[scale=.5]
\tikzset{VertexStyle/.style = {shape= circle,
ball color = red,inner sep= 0pt,outer sep= 0pt,minimum size= 7 pt}}   
\Vertex[LabelOut,Lpos=180,x=2 ,y=2,L=$\vec{0}$]{1}
   \Vertex[x=5 ,y=2,L=$$]{2}
   \Vertex[x=4 ,y=4,L=$$]{3}
   \Vertex[LabelOut,Lpos=0,x=7 ,y=4,L=$\vec{1}$]{4}
   \tikzset{VertexStyle/.style = {shape= circle,
ball color =  black,inner sep= 0pt,outer sep= 0pt,minimum size=6 pt}}     
   \Vertex[LabelOut,Lpos=180,x=4.5 ,y=6,L=$v$]{5}
      \Edge[style={ultra thick},color=red](1)(2)  \Edge[style={ultra thick},color=red](2)(4)  \Edge[style={ultra thick},color=red](3)(4)
    \Edge[style={ultra thick},color=red](1)(3)       
   \Edge(1)(5)
      \Edge(4)(5)
     \node  [color=blue] at (4,0) {$K^{\{2\}}$};
  \end{tikzpicture}
  \qquad\quad
  \begin{tikzpicture}[scale=.5]
\tikzset{VertexStyle/.style = {shape= circle,
ball color = red,inner sep= 0pt,outer sep= 0pt,minimum size= 7 pt}}   
\Vertex[LabelOut,Lpos=180,x=2 ,y=2,L=$\vec{0}$]{1}
   \Vertex[x=5 ,y=2,L=$$]{2}
   \Vertex[x=4 ,y=4,L=$$]{3}
   \Vertex[x=7 ,y=4,L=$$]{4}
  \Vertex[x=2 ,y=5,L=$$]{5}
   \Vertex[x=5 ,y=5,L=$$]{6}
   \Vertex[x=4 ,y=7,L=$$]{7}
   \Vertex[LabelOut,Lpos=30,x=7 ,y=7,L=$\vec{1}$]{8}
      \tikzset{VertexStyle/.style = {shape= circle,
ball color =  black,inner sep= 0pt,outer sep= 0pt,minimum size=6 pt}}  
    \Vertex[LabelOut,Lpos=180,x=9 ,y=4,L=$v$]{9}
      \Edge[style={ultra thick},color=red](1)(2)  \Edge[style={ultra thick},color=red](2)(4)  \Edge[style={ultra thick},color=red](3)(4)
    \Edge[style={ultra thick},color=red](1)(3)       
   \Edge[style={ultra thick},color=red](1)(5)
      \Edge[style={ultra thick},color=red](2)(6)
       \Edge[style={ultra thick},color=red](3)(7)
        \Edge[style={ultra thick},color=red](4)(8)
   \Edge[style={ultra thick},color=red](5)(6)
    \Edge[style={ultra thick},color=red](7)(8)
     \Edge[style={ultra thick},color=red](5)(7)
     \Edge[style={ultra thick},color=red](6)(8)
          \tikzset{EdgeStyle/.style = {-,bend right=40}}
      \Edge(1)(9)
      \Edge(9)(8)
     \node  [color=blue] at (4,0) {$K^{\{3\}}$};
  \end{tikzpicture}
  \end{center}

 \end{example}
 
 \begin{example}\label{k23dplus1examples}$K_{23}^{\{d\}}, d = 2,3$.
\medskip
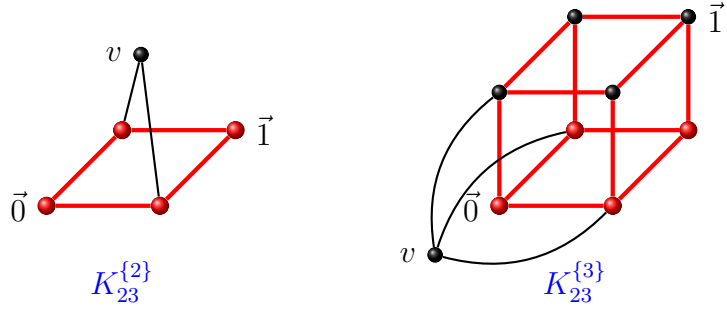
\begin{center}
   \begin{tikzpicture}[scale=.5]
\tikzset{VertexStyle/.style = {shape= circle,
ball color = red,inner sep= 0pt,outer sep= 0pt,minimum size= 7 pt}}   
\Vertex[LabelOut,Lpos=180,x=2 ,y=2,L=$\vec{0}$]{1}
   \Vertex[x=5 ,y=2,L=$$]{2}
   \Vertex[x=4 ,y=4,L=$$]{3}
   \Vertex[LabelOut,Lpos=0,x=7 ,y=4,L=$\vec{1}$]{4}
   \tikzset{VertexStyle/.style = {shape= circle,
ball color =  black,inner sep= 0pt,outer sep= 0pt,minimum size=6 pt}}     
   \Vertex[LabelOut,Lpos=180,x=4.5 ,y=6,L=$v$]{5}
      \Edge[style={ultra thick},color=red](1)(2)  \Edge[style={ultra thick},color=red](2)(4)  \Edge[style={ultra thick},color=red](3)(4)
    \Edge[style={ultra thick},color=red](1)(3)       
   \Edge(2)(5)
      \Edge(3)(5)
     \node  [color=blue] at (4,0) {$K_{23}^{\{2\}}$};
  \end{tikzpicture}
  \qquad\quad
  \begin{tikzpicture}[scale=.5]
\tikzset{VertexStyle/.style = {shape= circle,
ball color = red,inner sep= 0pt,outer sep= 0pt,minimum size= 7 pt}}   
\Vertex[LabelOut,Lpos=180,x=2 ,y=2,L=$\vec{0}$]{1}
   \Vertex[x=5 ,y=2,L=$$]{2}
   \Vertex[x=4 ,y=4,L=$$]{3}
   \Vertex[x=7 ,y=4,L=$$]{4}
   \tikzset{VertexStyle/.style = {shape= circle,
ball color =  black,inner sep= 0pt,outer sep= 0pt,minimum size=6 pt}}     
  \Vertex[x=2 ,y=5,L=$$]{5}
   \Vertex[x=5 ,y=5,L=$$]{6}
   \Vertex[x=4 ,y=7,L=$$]{7}
   \Vertex[LabelOut,Lpos=0,x=7 ,y=7,L=$\vec{1}$]{8}
     \Vertex[LabelOut,Lpos=180,x=.3 ,y=.7,L=$v$]{9}
      \Edge[style={ultra thick},color=red](1)(2)  \Edge[style={ultra thick},color=red](2)(4)  \Edge[style={ultra thick},color=red](3)(4)
    \Edge[style={ultra thick},color=red](1)(3)       
   \Edge[style={ultra thick},color=red](1)(5)
      \Edge[style={ultra thick},color=red](2)(6)
       \Edge[style={ultra thick},color=red](3)(7)
        \Edge[style={ultra thick},color=red](4)(8)
   \Edge[style={ultra thick},color=red](5)(6)
    \Edge[style={ultra thick},color=red](7)(8)
     \Edge[style={ultra thick},color=red](5)(7)
     \Edge[style={ultra thick},color=red](6)(8)
        \tikzset{EdgeStyle/.style = {-,bend right}}
     \Edge(9)(2)
        \tikzset{EdgeStyle/.style = {-,bend left}}
     \Edge(9)(3)
     \tikzset{EdgeStyle/.style = {-,bend left}}
     \Edge(9)(5)
     \node  [color=blue] at (4,0) {$K_{23}^{\{3\}}$};
  \end{tikzpicture}
  \end{center}
 \end{example}

\begin{lemma}\label{k23lemma}
For the neighborhoods $K^{\{d\}}$ and $K_{23}^{\{d\}}$ defined above, we have:
\begin{enumerate}
\item
$K^{\{d\}}$ is monotolerant for $d=1,2$ but monophobic for $d\ge 3$.
\item
$K_{23}^{\{d\}}$ is monotolerant for all $d\ge 2$.
\end{enumerate}
\end{lemma}
\begin{proof}
Monotolerance of $K^{\{1\}}$ and $K^{\{2\}}\approx K_{23}^{\{2\}}$ has been noted earlier. 
We will first demonstrate that $K^{\{3\}}$ is monophobic, by an argument that generalizes easily to dimensions $d>3$. Suppose
we have a 4-cube $\sigma$ which is monotolerant witness for $K^{\{3\}}$, with bottom face (following $H$) labeled as shown:
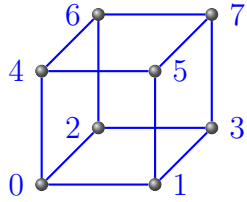
\begin{center}
\begin{tikzpicture}[scale=1.5]
\tikzset{VertexStyle/.style = {shape= circle,
ball color = gray,text= black,inner sep= 0pt,outer sep= 0pt,minimum size= 5pt}} 
   \Vertex[LabelOut,Lpos=180,x=0 ,y=0,L=${\color{blue}0}$]{1}
     \Vertex[LabelOut,Lpos=0,x=1 ,y=0,L=${\color{blue}1}$]{2}
       \Vertex[LabelOut,Lpos=180,x=.5 ,y=.5,L=${\color{blue}2}$]{3}
         \Vertex[LabelOut,Lpos=0,x=1.5 ,y=.5,L=${\color{blue}3}$]{4}
            \Vertex[LabelOut,Lpos=180,x=0 ,y=1,L=${\color{blue}4}$]{5}
               \Vertex[LabelOut,Lpos=0,x=1 ,y=1,L=${\color{blue}5}$]{6}
                  \Vertex[LabelOut,Lpos=180,x=.5 ,y=1.5,L=${\color{blue}6}$]{7}
                     \Vertex[LabelOut,Lpos=0,x=1.5 ,y=1.5,L=${\color{blue}7}$]{8}
    \Edge[style={thick},color=blue](1)(2)
      \Edge[style={thick},color=blue](3)(4)
        \Edge[style={thick},color=blue](1)(3)
          \Edge[style={thick},color=blue](2)(4)
           \Edge[style={thick},color=blue](5)(6)
      \Edge[style={thick},color=blue](7)(8)
        \Edge[style={thick},color=blue](5)(7)
          \Edge[style={thick},color=blue](6)(8)
          \Edge[style={thick},color=blue](1)(5)
      \Edge[style={thick},color=blue](2)(6)
        \Edge[style={thick},color=blue](3)(7)
          \Edge[style={ thick},color=blue](4)(8)
    \end{tikzpicture}
    \end{center}
The vertex $v$ must appear somewhere on the top face of $\sigma$, since otherwise $\sigma$ is a singular cube for the ordinary grid graph
$I_8^3$, contradicting Theorem \ref{maintheorem}. Furthermore, $v$ is linked only to vertices $\vec{0}$ and $\vec{1}$ in $G$, and thus
can appear only at one (or both) of the positions corresponding to $0$ and $7$ on the bottom face. If it appears only once, then $\sigma$ is 
a singular cube for a subgraph of $I_8^3$, again contradicting Theorem \ref{maintheorem}. If it appears twice, then the path
of length $3$  between the two occurrences of $v$ in in $\sigma$ induces a cycle of length $\le3$ containing $v$ in $K^{\{3\}}$.  But 
the smallest cycle in $K^{\{3\}}$ containing $v$ has length $5$, a contradiction.  This argument generalizes easily to dimensions $d>3$.

To prove that $K_{23}^{\{d\}}$ is monotolerant for $d\ge 3$ we first construct a monotolerant witness for $d=3$ and then explain how to 
generalize the pattern to all dimensions. For $d=3$, suppose that the vertices of $H$ are labeled as above, and define 
a singular $4$-cube $\sigma$
as follows:

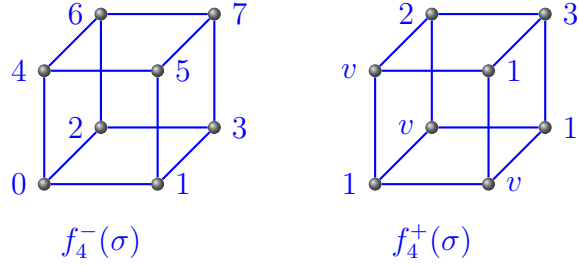
\begin{center}
\begin{tikzpicture}[scale=1.5]
\tikzset{VertexStyle/.style = {shape= circle,
ball color = gray,text= black,inner sep= 0pt,outer sep= 0pt,minimum size= 5pt}} 
   \Vertex[LabelOut,Lpos=180,x=0 ,y=0,L=${\color{blue}0}$]{1}
     \Vertex[LabelOut,Lpos=0,x=1 ,y=0,L=${\color{blue}1}$]{2}
       \Vertex[LabelOut,Lpos=180,x=.5 ,y=.5,L=${\color{blue}2}$]{3}
         \Vertex[LabelOut,Lpos=0,x=1.5 ,y=.5,L=${\color{blue}3}$]{4}
            \Vertex[LabelOut,Lpos=180,x=0 ,y=1,L=${\color{blue}4}$]{5}
               \Vertex[LabelOut,Lpos=0,x=1 ,y=1,L=${\color{blue}5}$]{6}
                  \Vertex[LabelOut,Lpos=180,x=.5 ,y=1.5,L=${\color{blue}6}$]{7}
                     \Vertex[LabelOut,Lpos=0,x=1.5 ,y=1.5,L=${\color{blue}7}$]{8}
    \Edge[style={thick},color=blue](1)(2)
      \Edge[style={thick},color=blue](3)(4)
        \Edge[style={thick},color=blue](1)(3)
          \Edge[style={thick},color=blue](2)(4)
           \Edge[style={thick},color=blue](5)(6)
      \Edge[style={thick},color=blue](7)(8)
        \Edge[style={thick},color=blue](5)(7)
          \Edge[style={thick},color=blue](6)(8)
          \Edge[style={thick},color=blue](1)(5)
      \Edge[style={thick},color=blue](2)(6)
        \Edge[style={thick},color=blue](3)(7)
          \Edge[style={ thick},color=blue](4)(8)
             \node  [color=blue] at (.5,-.5) {$f_4^{-}(\sigma)$};
    \end{tikzpicture}
    \qquad
    \begin{tikzpicture}[scale=1.5]
\tikzset{VertexStyle/.style = {shape= circle,
ball color = gray,text= black,inner sep= 0pt,outer sep= 0pt,minimum size= 5pt}} 
   \Vertex[LabelOut,Lpos=180,x=0 ,y=0,L=${\color{blue}1}$]{1}
     \Vertex[LabelOut,Lpos=0,x=1 ,y=0,L=${\color{blue}v}$]{2}
       \Vertex[LabelOut,Lpos=180,x=.5 ,y=.5,L=${\color{blue}v}$]{3}
         \Vertex[LabelOut,Lpos=0,x=1.5 ,y=.5,L=${\color{blue}1}$]{4}
            \Vertex[LabelOut,Lpos=180,x=0 ,y=1,L=${\color{blue}v}$]{5}
               \Vertex[LabelOut,Lpos=0,x=1 ,y=1,L=${\color{blue}1}$]{6}
                  \Vertex[LabelOut,Lpos=180,x=.5 ,y=1.5,L=${\color{blue}2}$]{7}
                     \Vertex[LabelOut,Lpos=0,x=1.5 ,y=1.5,L=${\color{blue}3}$]{8}
    \Edge[style={thick},color=blue](1)(2)
      \Edge[style={thick},color=blue](3)(4)
        \Edge[style={thick},color=blue](1)(3)
          \Edge[style={thick},color=blue](2)(4)
           \Edge[style={thick},color=blue](5)(6)
      \Edge[style={thick},color=blue](7)(8)
        \Edge[style={thick},color=blue](5)(7)
          \Edge[style={thick},color=blue](6)(8)
          \Edge[style={thick},color=blue](1)(5)
      \Edge[style={thick},color=blue](2)(6)
        \Edge[style={thick},color=blue](3)(7)
          \Edge[style={ thick},color=blue](4)(8)
             \node  [color=blue] at (.5,-.5) {$f_4^{+}(\sigma)$};
    \end{tikzpicture}
    \end{center}
Then $\sigma$ is a singular cube for $K^{\{3\}}$ and is monotolerant since every face (except $f_4^{-}(\sigma)$) contains a vertex
labeled by $v$. 

To obtain a monotolerant witness for $K_{23}^{\{d\}}$, $d>3$, construct a singular $d+1$ cube $\sigma$ as follows. Let $f_{d+1}^{-}(\sigma)$ be a $d$-cube with vertices
labeled by $0,\dots,2^d-1$ as above. To label upper face $f_{d+1}^{+}(\sigma)$, first assign label $1$ to the vertex opposite $0$; then assign label $v$ to all of the vertices opposite $1, 2, 4, \dots, 2^{d-1}$.
Finally, for each remaining vertex labeled $k$ in $f_{d+1}^{-}(\sigma)$, write $k$ in binary notation, drop the last two binary digits, and assign that label to the vertex opposite $k$. This pattern 
produces  a singular cube $\sigma$ for  $K_{23}^{\{d\}}$ and every face except $f_{d+1}^{-}(\sigma)$ contains a vertex labeled $v$.
\end{proof}

It is easy to see that removing any of the edges outside $H$ renders $K_{23}^{\{d\}}$ monophobic, hence it is a 
minimal monotolerant neighborhood for dimension $d$. 

\section{Natural Cycles}

In the examples presented in Section 2, we used the term ``natural cycle'' to mean a cycle in $C_d(G)$ whose terms correspond to terms of 
a cycle in a cellular complex whose existence is evident geometrically. A corresponding cycle in $C_d(G)$ is usually easy to construct in small cases, but can require considerable effort since 
geometric $d$-cubes correspond to injective singular $d$-cubes in $2^d d!$ ways.  

In this section we again consider graphs $G_K$ arising from simple quadrangulations $K$ of a $2$-sphere (defined in Section 5), and  prove that every cellular $2$-cycle in $K$ corresponds to a cycle of singular $2$-cubes in $C_2(G_K)$
with the same coefficients. We emphasize that no claim is being made about the nontriviality of the cycle, only its existence.

\begin{theorem}{\rm (Existence of Natural 2-Cycles)}\label{natcycles}
Suppose that $G_K$ is derived from a simple quadrangulation $K$ of $S^2$.
 Let  $V(G_K) = K_0$ and $ E(G_K) = K_1$, and let $C_n(G_K)$ denote the $Z$-module spanned by singular cubes $\sigma: I_1^n \to G_K$. Then, if $Z = \sum_{i=1}^m \alpha_i z_i$  is a $2$-cycle in $K$ ($\alpha_i \in \ZZ,z_i \in K_2, \partial Z = 0$),
there exists a cycle $\tilde{Z} = \sum_{i=1}^m \alpha_i\tilde{z}_i$ in $C_2(G_K)$ such that for each $i$, $\tilde{z}_i$ is a singular $2$-cube with support equal to the set of vertices of $z_i$.
\end{theorem}

\begin{proof}{\rm (Sketch)}
Consider the dual graph $G^*_K$ of $G_K$, whose vertices are the 2-cells of $K$ and whose edges are adjacent pairs of 2-cells. This graph partitions $S^2$ into regions corresponding to the original 0-cells of $K$, in which each vertex has degree 4.  Elementary graph theory shows that we may 
cover the edges of $G^*_K$ by a collection of disjoint directed cycles in which edges entering a vertex always leave by the opposite edge (a local condition). Now associate to every vertex $z \in V(G^*_K)$ a singular cube $\tilde{z}$ according to the following rule,  which picks a  vertex to ``anchor'' the cube among four possible orientation-preserving rotations.
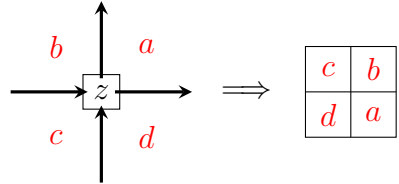
\begin{center}
\begin{tikzpicture}[>=stealth,scale=.6]
\draw[->, very thick] (0,.3) -> (0,2) node[above]{};
\draw[->,very  thick] (.3,0) -> (2,0) node[right]{};
\draw[->,very  thick] (-2,0) -> (-.3,0) node[right]{};
\draw[->, very thick] (0,-2) -> (0,-.3) node[right]{};
\node [rectangle, draw] at (0,0) {$z$};
\node at (1,1) {\color{red}$a$};
\node at (-1,1) {\color{red}$b$};
\node at (-1,-1) {\color{red}$c$};
\node at (1,-1) {\color{red}$d$};
\node at (3.2,0) {$\Longrightarrow $};
\draw (4.5,-1) rectangle (6.5,1);
\draw (4.5,0) -- (6.5,0);
\draw (5.5,-1) -- (5.5,1);
\node at (6,-.5) {\color{red}$a$};
\node at (6,.5) {\color{red}$b$};
\node at (5,.5) {\color{red}$c$};
\node at (5,-.5) {\color{red}$d$};
\end{tikzpicture}
\end{center}
This rule associates every cell in $K$ to a singular $2$-cube, and it is straightforward to check,
examining all possibilities, that boundary edges of adjacent cells always cancel (they have the same orientation but 
opposite sign).
 Hence  $\tilde{Z} = \sum_{i=1}^m \alpha_i\tilde{z}_i$ is a cycle in $C_2(G_k)$, as claimed.
\end{proof}
Figure (\ref{xxx}) above illustrates steps in the construction for the 10-vertex crown, regarded as the 2-skeleton of a quadrangular 2-complex, 
from its corresponding dual graph. Note that the cycle of singular cubes differs (uniformly by rotations) from the one presented 
in Example \ref{ex3.3}.
Since every rigid cube in $G_K$ is monophobic, we may conclude (as was done in \cite{BGJW1}) that the cycle is nontrivial.
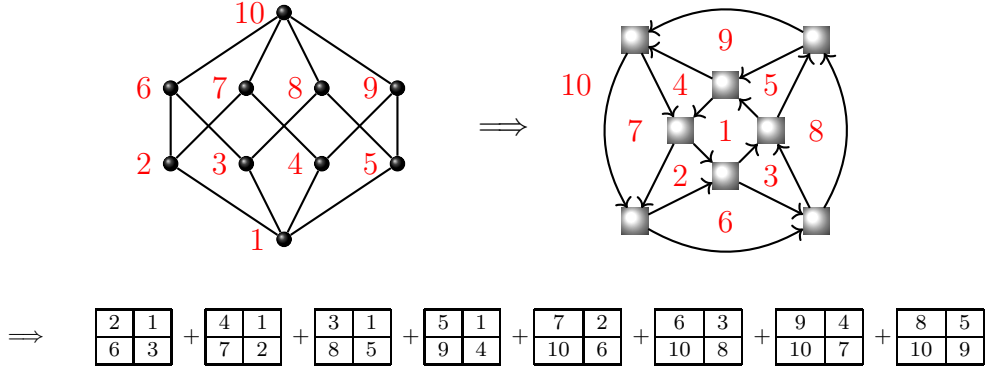
\begin{figure}\label{dualgraph}
\medskip
\begin{center}
\begin{tikzpicture}[scale=1]
\tikzset{VertexStyle/.style = {shape= circle,shading = none,
ball color = black,text= black,inner sep= 0pt,outer sep= 0pt,minimum size= 6 pt}}
   \Vertex[LabelOut,Lpos=180,x=1.5,y=0,L=${\color{red}1}$]{1}
   \Vertex[LabelOut,Lpos=180,x=0 ,y=1,L=${\color{red}2}$]{2}
   \Vertex[LabelOut,Lpos=180,x=1,y=1,L=${\color{red}3}$]{3}
    \Vertex[LabelOut,Lpos=180,x=2,y=1,L=${\color{red}4}$]{4}
     \Vertex[LabelOut,Lpos=180,x=3,y=1,L=${\color{red}5}$]{5}
      \Vertex[LabelOut,Lpos=180,x=0,y=2,L=${\color{red}6}$]{6}
      \Vertex[LabelOut,Lpos=180,x=1,y=2,L=${\color{red}7}$]{7}
      \Vertex[LabelOut,Lpos=180,x=2,y=2,L=${\color{red}8}$]{8}
          \Vertex[LabelOut,Lpos=180,x=3,y=2,L=${\color{red}9}$]{9}
                 \Vertex[LabelOut,Lpos=180,x=1.5,y=3,L=${\color{red}10}$]{10}
    \Edge[style={thick},color=black](1)(2)
    \Edge[style={thick},color=black](1)(3)  \Edge[style={thick},color=black](1)(4)  \Edge[style={thick},color=black](1)(5)  
    \Edge[style={thick},color=black](2)(6) \Edge[style={thick},color=black](2)(7) \Edge[style={thick},color=black](3)(8)\Edge[style={thick},color=black](3)(6)
    \Edge[style={thick},color=black](4)(7) \Edge[style={thick},color=black](4)(9)\Edge[style={thick},color=black](5)(8)\Edge[style={thick},color=black](5)(9)
    \Edge[style={thick},color=black](6)(10)\Edge[style={thick},color=black](7)(10)\Edge[style={thick},color=black](8)(10)\Edge[style={thick},color=black](9)(10)
    \end{tikzpicture}
    \qquad
    \raisebox{4em}{$\Longrightarrow$}
    \raisebox{0em}{
    \begin{tikzpicture}[scale=.6]
\tikzset{VertexStyle/.style = {shape= rectangle,shading = none,
ball color = white,text= white,inner sep= 5pt,outer sep= 0pt,minimum size= 6 pt}}
   \Vertex[LabelOut,Lpos=180,x=0,y=2,L=$$]{1}
   \Vertex[LabelOut,Lpos=180,x=4 ,y=2,L=$$]{2}
   \Vertex[LabelOut,Lpos=180,x=2,y=3,L=$$]{3}
    \Vertex[LabelOut,Lpos=180,x=1,y=4,L=$$]{4}
     \Vertex[LabelOut,Lpos=180,x=3,y=4,L=$$]{5}
      \Vertex[LabelOut,Lpos=180,x=2,y=5,L=$$]{6}
      \Vertex[LabelOut,Lpos=180,x=0,y=6,L=$$]{7}
      \Vertex[LabelOut,Lpos=180,x=4,y=6,L=$$]{8}
    \Edge[style={thick,->,bend right},color=black](1)(2)
    \Edge[style={thick,->},color=black](1)(3)  
       \Edge[style={thick,->},color=black](3)(2) 
     \Edge[style={thick,->},color=black](4)(1)  
    \Edge[style={thick,->,bend right},color=black](7)(1)  
     \Edge[style={thick,->},color=black](4)(3)  
        \Edge[style={thick,->},color=black](2)(5)  
    \Edge[style={thick,->},color=black](3)(5)  
     \Edge[style={thick,->},color=black](6)(4)  
    \Edge[style={thick,->},color=black](5)(6) 
     \Edge[style={thick,->},color=black](5)(8)
      \Edge[style={thick,->},color=black](7)(4)
         \Edge[style={thick,->},color=black](6)(7)
      \Edge[style={thick,->},color=black](8)(6)
       \Edge[style={thick,->,bend right},color=black](2)(8)
        \Edge[style={thick,->,bend right},color=black](8)(7)
                         \node  [color=red] at (2,2) {$6$};
          \node  [color=red] at (2,4) {$1$};
 \node  [color=red] at (2,6) {$9$};
 \node  [color=red] at (3,3) {$3$};
 \node  [color=red] at (1,3) {$2$};
 \node  [color=red] at (0,4) {$7$};
  \node  [color=red] at (4,4) {$8$};
    \node  [color=red] at (1,5) {$4$};
      \node  [color=red] at (3,5) {$5$};
        \node  [color=red] at (-1.3,5) {$10$};
  \end{tikzpicture}
  }
  \end{center}
 
 {\tiny
 \[
 \Longrightarrow\qquad
 \begin{array}{|c|c|}
 \hline
 2&1\\
 \hline
6&3\\
 \hline
 \end{array}\;+
  \begin{array}{|c|c|}
 \hline
 4&1\\
 \hline
 7&2\\
 \hline
\end{array}
\;+
  \begin{array}{|c|c|}
 \hline
 3&1\\
 \hline
 8&5\\
 \hline
\end{array}
\;+
  \begin{array}{|c|c|}
 \hline
 5&1\\
 \hline
 9&4\\
 \hline
\end{array}
\;+
 \begin{array}{|c|c|}
 \hline
 7&2\\
 \hline
 10&6\\
 \hline
\end{array}
\;+
 \begin{array}{|c|c|}
 \hline
 6&3\\
 \hline
 10&8\\
 \hline
\end{array}
\;+
 \begin{array}{|c|c|}
 \hline
 9&4\\
 \hline
 10&7\\
 \hline
\end{array}
\;+
 \begin{array}{|c|c|}
 \hline
8&5\\
 \hline
 10&9\\
 \hline
\end{array}
\]
}\caption{Constructing a singular 2-cycle from $G_K$ via the dual graph of $K$.}
\label{xxx}
\end{figure}

It is tempting to conjecture that an analog of Theorem \ref{natcycles} holds in all dimensions. By analogy with the cellular case, 
one might expect that, 
for a ``simple'' cubulation $K$ of $S^d$ there is a cycle in $C_d(G_K)$ 
with all coefficients equal to $+1$, in which singular
$d$-cubes correspond directly to $d$-cells in $K$. 
However,  this result appears to be false in dimension 3, even for an example as simple as 
the cellular decomposition of $\partial Q_4$ by its $3$-faces:  computation shows that
none of the $48^8$ possible orientations of $3$-faces yields a positive sum in $C_3(\partial Q_4)$ with boundary zero. 

Notwithstanding the above, there are clearly cycles with coefficients $\pm 1$ in $C_3(Q_4)$, obtained, 
for example, from the discrete boundary of $Q_4$. 
More generally, boundary cycles obtained from $Q_d, d \ge 2$, can be used as a template for 
constructing  ``natural" $d$-cycles in $C_d(G_m^d)$  for all $ d > 1, m \ge 3$, with all
coefficients $\pm 1$.  Here $G_m^d$  denotes the subdivided cube, defined in Example \ref{ex3.1}.  Since   
these cycles contain terms supported by cubes with monophobic neighborhoods (Theorem \ref{gridlemma}),  it follows from Theorem \ref{maintheorem} (as claimed in Section 2) that these cycles are nontrivial in all dimensions.

Some of the issues discussed above have been overshadowed by recent work computing homology groups
explicitly, thus proving much more than nontriviality. 
Specifically, \cite{CK26} computes the homology of subdivided cubes for many $m \ge 2,d\ge 3$ and shows that $H_d(G_m^d) \approx \ZZ$.  
Also, combining Theorem \ref{suffcondition} with  results in \cite{BJ26}  
proves that if $G_K$ is the 1-skeleton of a simple quadrangulation $K$ of $S^2$, then 
$H_2(G_K) \cong H_2^{\mathrm Cell}(K)$ provided $K_{23}$  and $Q_3$ do not occur in $G_K$.  
Despite these important results, several questions in the spirit of this section may remain of interest. We propose the following:
\begin{problem}\label{natquests}
For $d>2$, define a cubulation $K$ of $S^d$ to be {\em simple} if the $1$-skeleton of every $d$-cell is isomorphic to $Q_d$. 
Then:
\begin{enumerate} \item For which simple cubulations of $S^d$, is every rigid $d$-cube in $G_K$ monophobic?
\item For a simple cubulation of $S^d$, is there always a cycle in $C_d(G_K)$ in which every cell is represented by a singular $d$-cube
with coefficient $\pm 1$?  
\item
Under what conditions does a simple cubulation of $S^d$ have a cycle with all coefficients $+1$?
\end{enumerate}
\end{problem}

\section{Possible Future Directions}

Our identification of monophobicity in Examples \ref{ex3.3}, \ref{ex3.1} and \ref{ex3.2} can be seen as first steps in a program of identifying graphical analogs of geometric spheres. When this project began, the exact determination of discrete homology and homotopy groups for graphs as small as those illustrated in Examples  \ref{ex3.3}, \ref{ex3.1}, and \ref{ex3.2} was only barely within the scope of existing computational methods.  As noted above, much progress has been made recently, e.g. \cite{KK24}, and especially \cite{CK26} which includes many explicit computations of homotopy groups of  suspensions and subdivided cubes in higher dimensions. 

Continuing this program, it seems desirable to find general neighborhood conditions (in all dimensions) with  properties similar to monophobicity, i.e. forcing the nontriviality of local homology and homotopy. A  step (albeit small) in this direction would be extending Theorem \ref{dim2theorem} to higher dimensions.  Although Lemma \ref{k23lemma} may be a promising start, computations in dimension 3 reveal many other minimal monotolerant examples, and a classification as simple as Theorem \ref{dim2theorem} may not be possible. 

Example \ref{ex3.2} hints at the possibility of extending all of the results in this paper to configurations involving both quadrilaterals and triangles in dimension $2$ and generalizations in higher dimensions. This would of course require a revised and more inclusive definition of ``monophobicity" for such cells.

Finally,  it will be extremely interesting to identify the exact relationship between our characterization of minimal monotolerant 2-neighborhoods (Theorem \ref{dim2theorem}) and the more general computational tools for dimension 2 recently developed by Ender and Kapulkin (\cite{EK25}, \cite{EK26}). Their results include a list of ``basic 3-shapes'' for computation that contains (as graphs) the five neighborhoods listed in our Theorem \ref{dim2theorem}. The similarity of these lists is certainly more than coincidental.

\section{Acknowledgments}

The authors are grateful for advice and significant contributions from Eric Babson, in particular for the statement and proof of Theorem \ref{maintheorem}(2), and for the dual graph formulation
of the proof of Theorem \ref{natcycles}.
We also acknowledge helpful input from H\'{e}l\`{e}ne  Barcelo, Abdul Jarrah, Volkmar Welker,  Georg Wille, Chris Kapulkin,  Daniel Carranza, Chris Staecker, and Danish Ali.
   
\bibliographystyle{plain}
\bibliography{CGBib.bib}
\end{document}